\documentclass[11pt,reqno,a4paper]{amsart}
\usepackage{color,enumerate,cancel,hyperref,mathtools,tikz-cd,amssymb,dsfont,mathrsfs}
\newtheorem{theorem}{Theorem}[section]
\newtheorem{lemma}[theorem]{Lemma}
\newtheorem{problem}[theorem]{Problem}
\newtheorem{proposition}[theorem]{Proposition}
\newtheorem{corollary}[theorem]{Corollary}

\theoremstyle{definition}
\newtheorem{definition}[theorem]{Definition}

\usepackage{enumitem}

\theoremstyle{remark}
\newtheorem{remark}[theorem]{Remark}
\numberwithin{equation}{section}

\begin{document}

\title[Derivatives of Tensor Products and Applications to Spaces $C(K^n,X)$]{Derivatives of Tensor Products and Applications to Spaces $C(K^n,X)$}
\author{Leandro Candido}
\address{Universidade Federal de S\~ao Paulo - UNIFESP. Instituto de Ci\^encia e Tecnologia. Departamento de Matem\'atica. S\~ao Jos\'e dos Campos - SP, Brasil}
\email{\texttt{leandro.candido@unifesp.br}}
\thanks{ The author was supported by Funda\c c\~ao de Amparo \`a Pesquisa do Estado de S\~ao Paulo - FAPESP No. 2023/12916-1 }




\begin{abstract}
In this paper, we develop an abstract theory of derivatives for Banach
spaces based on objects that we call \emph{bidual assignments}. This
framework encompasses both the Semadeni derivative and the recently
introduced Semadeni--Pe{\l}czy\'nski derivative. More generally,
suitable ideals of subsets of dual spaces give rise to a broad family
of derivatives within this setting.

We establish direct-sum and tensor-product formulas for these
derivatives, showing that they behave naturally with respect to direct
sums and injective tensor products. We then obtain explicit
descriptions of derivatives associated with compact trees and study
the corresponding questions for finite products of compact lines.

As the main application, we establish classification results for spaces
of the form $C(K^n,X)$. In particular, for uncountable ordinals
$\alpha$ and $\beta$, an integer $n\geq1$, and Banach spaces $X$
satisfying suitable rigidity assumptions, we prove that
\[
C([0,\alpha]^n,X)\sim C([0,\beta]^n,X)
\quad\text{if and only if}\quad
C([0,\alpha])\sim C([0,\beta]).
\]
This extends Kislyakov's classification of the spaces
$C([0,\alpha])$ and its vector-valued extension due to Galego to finite
powers of ordinal intervals.
\end{abstract}

\maketitle

\section{Introduction}

The notion of a derivative of a Banach space, as understood in the present
work, originates in a seminal paper of Semadeni \cite{Se2}. Given a Banach
space $X$, Semadeni considered the subspace $\Delta(X)\subset X^{**}$
consisting of all weak$^*$-sequentially continuous functionals on $X^*$,
and observed that the quotient
\[
\mathcal S(X)=\Delta(X)/X
\]
is an isomorphic invariant of $X$. He computed
\[
\mathcal S(C([0,\omega_1]))\sim\mathbb R
\qquad\text{and}\qquad
\mathcal S(C([0,\omega_1])\oplus C([0,\omega_1]))
\sim
\mathbb R^2,
\]
thereby proving that $C([0,\omega_1])$ is not isomorphic to its square. This
provided the first example of a Banach space $X$ satisfying $X\not\sim X\oplus X$,
solving a problem posed by Banach \cite{Banach}.

Semadeni's ideas were subsequently developed by Kislyakov
\cite{Kislyakov}, who employed related derivatives to obtain a classification
of the spaces $C([0,\alpha])$, where $\alpha$ is an uncountable ordinal.
Later, Galego extended these methods to the vector-valued setting,
obtaining several classification results for spaces of the form $C(K,X)$,
where $K$ is an ordinal interval and $X$ is a Banach space; see, for
instance, \cite{galego1,galego2,galego3,galego4}.

More recently, Korpalski \cite{KorpalskiThesis,Korpalski} introduced the
Semadeni--Pe{\l}czy\'nski derivative and studied it systematically in the
setting of spaces of continuous functions on compact lines. In particular, he
applied this derivative to spaces of the form $C(K^n)$, where $K$ is a compact
line. His work builds on ideas originating in Semadeni's work and further
developed in \cite{CandidoSquares}.

The present paper is motivated by the observation that all the constructions
mentioned above share a common underlying structure. We introduce an abstract notion of derivative based on what we call a
\emph{bidual assignment} (Definition~\ref{Def:BidualAssignment}), that
is, a rule assigning to each Banach space $X$ a closed subspace
\[
X\subset \Delta(X)\subset X^{**}
\]
which behaves functorially with respect to bounded linear operators.
Associated with such an assignment, we consider the derivative
\[
\mathcal{S}_\Delta(X)=\Delta(X)/X,
\]
which provides an isomorphic invariant of the Banach space $X$.

One of the main sources of examples in this paper is provided by what we call
\emph{testing ideals} $\mathfrak I$ (Definition~\ref{Def:Test}), namely ideals of subsets of dual spaces
satisfying suitable stability properties. Such ideals induce bidual
assignments $\Delta_{\mathfrak I}$, and therefore derivatives
$\mathcal S_{\mathfrak I}$.

Among the main structural results of the paper are direct-sum and
tensor-product formulas for derivatives arising from bidual assignments
(Theorems~\ref{Thm:IsometryDerivative} and
\ref{Thm:TensorQuotientDerivative}). Under suitable assumptions, these
derivatives interact naturally with $c_0$-sums, $\ell_p$-sums, and
injective tensor products. We also investigate ideal-generated
assignments on spaces $C(K)$ and obtain a canonical function
representation for $\Delta_{\mathfrak I}(C(K))$ (Proposition~\ref{Prop:Fecho}). Under an additional
extension assumption, this representation allows us to identify the
associated derivative with a space of the form $C_0(U)$ (Theorem~\ref{Thm:ComplementationC0}).

We apply these structural results to compact trees and finite products
of compact lines. For compact trees, we explicitly compute the derivatives associated
with the testing ideals $\mathfrak I_{<\kappa}$ and use their iterates
to obtain isomorphic invariants for finite products of compact trees.

Our main application is Theorem~\ref{Thm:MainSuper}, a classification
result for spaces of the form $C([0,\alpha]^n,X)$, where $\alpha$ is
an uncountable ordinal and $X$ is a Banach space satisfying suitable
rigidity assumptions. More precisely, we prove that
\[
C([0,\alpha]^n,X)\sim C([0,\beta]^n,X)
\quad\text{if and only if}\quad
C([0,\alpha])\sim C([0,\beta]),
\]
whenever $X$ is a Mazur space with the Gelfand--Phillips property,
contains no isomorphic copy of $c_0$, and satisfies
$X^p\not\sim X^q$ for distinct positive integers $p$ and $q$. This
extends Kislyakov's classification of the spaces $C([0,\alpha])$ and
its vector-valued extension due to Galego from ordinal intervals to
their finite powers.

The paper is organized as follows. In Section~\ref{Sec:Prelim}, we
introduce the terminology and notation used throughout the paper. In
Section~\ref{SeC:BidualAssignments}, we define bidual assignments and
their associated derivatives and study their behavior with respect
to direct sums. In Section~\ref{Sec:Injective}, we study the
interaction of these derivatives with injective tensor products and
introduce the class of testing ideals. In
Section~\ref{Sec:DerivativesC(K,Y)}, we investigate the hypotheses of
the tensor-product results for spaces of the form $C(K,Y)$. In
particular, we consider scattered compact spaces and finite products
of compact lines, obtain a function representation for
ideal-generated assignments on $C(K)$, and study the complementability
of $C(K)$ in $\Delta_{\mathfrak I}(C(K))$. Finally, in
Section~\ref{Sec:Applications}, we apply the general theory to the
testing ideals $\mathfrak I_\kappa$ and $\mathfrak I_{<\kappa}$,
compute derivatives for compact trees, and establish our classification results.

\section{Terminology and Preliminaries}
\label{Sec:Prelim}

In this section, we collect the basic terminology, notation, and
conventions used throughout the paper.

Given a topological space $K$, the character of a point $x\in K$,
denoted by $\chi(K,x)$, is the least cardinality of a local base at
$x$. The character and the density of $K$ are defined, respectively,
by $\chi(K)=\sup_{x\in K}\chi(K,x)$ and $\operatorname{dens}(K)= \min\{|D|:D\subset K\text{ is dense in }K\}$. 

We recall that a topological space is said to be Fréchet--Urysohn if, whenever a
point belongs to the closure of a subset $A$ of the space, there is a
sequence of points in $A$ converging to that point; see
\cite[Section~A,~p.~18]{EGT}.

We shall also consider compact trees endowed with the interval topology,
henceforth simply called compact trees. Recall that a tree $K$ is a partially
ordered set such that, for every $x\in K$, the set of predecessors $\operatorname{pred}(x)=\{y\in K:y<x\}$
is well ordered. The interval topology on $K$ is generated by sets of the form $(s,t]=\{x\in K:s<x\leq t\}$,
together with all singletons $\{x\}$ such that $x$ is a minimal element of
$K$. The height of a point $x\in K$ is the order type of
$\operatorname{pred}(x)$, that is, $\operatorname{ht}(x)=\operatorname{Ord}(\operatorname{pred}(x))$; see \cite{Tod}.

We now turn to Banach-space terminology. For Banach spaces $X$ and $Y$, we
write $X\sim Y$ whenever $X$ and $Y$ are linearly isomorphic. If $X$ and $Y$
are isometrically isomorphic, we write $X\cong Y$. If $T:X\to Y$ is an isomorphic embedding, its distortion is
$\|T\|\|T^{-1}\|$, where $T^{-1}$ is defined on $T(X)$. We say that
$X$ embeds into $Y$ with distortion at most $\lambda$ if there exists
such an operator $T$ satisfying $\|T\|\|T^{-1}\|\leq\lambda$.

For every Banach space $X$, $X^*$ denotes its topological dual and $X^{**}$
its bidual. We always identify $X$ with its canonical image in $X^{**}$.

If $Y$ is a closed subspace of $X$, the canonical quotient map of $X$ onto $X/Y$ will be
denoted by $x\mapsto [x]$.

We shall use direct sums determined by Banach spaces with unconditional
bases. Let $E$ be a Banach space with a normalized $1$-unconditional
basis $(e_\alpha)_{\alpha<\xi}$. Given a family of Banach spaces $(X_\alpha)_{\alpha<\xi}$, we define its $E$-direct sum by
\[
\big(\bigoplus_{\alpha<\xi}X_\alpha\big)_E
=\{(x_\alpha)_{\alpha<\xi}\in\prod_{\alpha<\xi}X_\alpha:\sum_{\alpha<\xi}\|x_\alpha\|e_\alpha\text{ converges in }E\},\]
where $\sum_{\alpha<\xi}\|x_\alpha\|e_\alpha$ is understood as the limit of the net $\big(\sum_{\alpha\in A}\|x_\alpha\|e_\alpha\big)_{A\in[\xi]^{<\omega}}$, and $[\xi]^{<\omega}$ is ordered by inclusion. The norm is given by
\[
\|(x_\alpha)_{\alpha<\xi}\|
=
\left\|
\sum_{\alpha<\xi}\|x_\alpha\|e_\alpha
\right\|_E.
\]
Equipped with this norm,
$\big(\bigoplus_{\alpha<\xi}X_\alpha\big)_E$ is a Banach space.

For a Banach space $X$, we say that $B_{X^*}$ is
weak$^*$-sequentially compact if every sequence in $B_{X^*}$ admits a
weak$^*$-convergent subsequence. Throughout this paper, if
$A\subset K\subset X^*$, we say that $A$ is weak$^*$-sequentially
dense in $K$ if the smallest subset of $K$ containing $A$ and
sequentially closed in the relative weak$^*$ topology is $K$.

A Banach space is said to be \emph{weakly compactly generated} ($\operatorname{WCG}$)
if it contains a weakly compact subset whose linear span is dense; see
\cite[Definition~13.1]{Fabian}.

A subset $A$ of a Banach space $X$ is said to be \emph{limited} if
every weak$^*$-null sequence in $X^*$ converges uniformly to $0$ on
$A$. A Banach space $X$ has the \emph{Gelfand--Phillips property} if
every limited subset of $X$ is relatively norm compact; see
\cite[p.~286]{Fabian}.

A Banach space $X$ has the approximation property
if, for every compact set $M\subset X$ and every $\varepsilon>0$,
there exists a finite-rank operator $T:X\to X$ such that
$\|Tx-x\|<\varepsilon$ for every $x\in M$.

We next recall some classical facts concerning spaces of continuous
functions. For a compact Hausdorff space $K$ and a Banach space $X$, $C(K,X)$ denotes the Banach space of all
continuous functions $f:K\to X$, endowed with the supremum norm. In the scalar case $X=\mathbb{R}$, we use the standard notation $C(K)$. By the Riesz representation theorem \cite[Theorem 18.4.1]{Se},
$C(K)^*$ can be identified isometrically with the Banach space $\mathcal M(K)$ of Radon measures on $K$, endowed with the variation norm. In particular, when $K$ is a scattered compact Hausdorff space, we use the canonical identifications $C(K)^*=\ell_1(K)$ and $C(K)^{**}=\ell_\infty(K)$. 

For Banach spaces $X$ and $Y$, we denote by $X\widehat{\otimes}_{\varepsilon}Y$ and $X\widehat{\otimes}_{\pi}Y$ their injective and projective tensor products, respectively. We shall repeatedly use the canonical isometric identification
\[
C(K)\widehat{\otimes}_{\varepsilon}Y \cong C(K,Y);
\]
see \cite{Ryan}. 

Finally, we recall the notion of an ideal of subsets. Let $S$ be a set.
An \emph{ideal} on $S$ is a family
$\mathcal I\subset\mathcal P(S)$ containing the empty set and closed
under taking subsets and finite unions.

\section{Bidual Assignments and Derivatives of Banach Spaces}
\label{SeC:BidualAssignments}

The aim of this section is to isolate a common framework underlying
several derivative constructions appearing in
\cite{CandidoSquares,Kislyakov,KorpalskiThesis,Korpalski,Se2} and in
Galego's classification theory for spaces of continuous vector-valued
functions \cite{galego1,galego2,galego3,galego4}. We first introduce
bidual assignments and study the basic properties of their associated
derivatives, including invariance under isomorphisms and their behavior
under embeddings. We then investigate their interaction with direct
sums and show that, in contrast, preservation of quotients cannot be
expected in general.

\begin{definition}\label{Def:BidualAssignment}
A \emph{bidual assignment} is a rule $\Delta$ which assigns to each Banach space $X$
a closed linear subspace $X \subset \Delta(X) \subset X^{**}$,
such that for every bounded linear operator $T:X\to Y$, one has
$T^{**}(\Delta(X))\subset \Delta(Y)$. This condition will be called
\emph{functoriality}.

For a bidual assignment $\Delta$, we define the \emph{derivative of $X$ with respect to $\Delta$} as the quotient space
\[
\mathcal{S}_{\Delta}(X)=\Delta(X)/X.
\]
\end{definition}

The simplest example is the trivial assignment, given by $\Delta(X)=X$ for every Banach space $X$. The principal nontrivial
examples motivating our framework are the following.

\medskip

\noindent
\textbf{The Semadeni assignment.}
\[
\Delta_{\mathfrak{s}}(X)
=
\{x^{**}\in X^{**}: x^{**}\text{ is weak}^*\text{-sequentially continuous}\},
\]
which gives rise to the Semadeni derivative $\mathcal{S}(X)$, see \cite{CandidoSquares}.

\medskip

\noindent
\textbf{The Korpalski assignment.}
For each infinite cardinal $\kappa$,
\[
\Delta_{\kappa}(X)
=
\big\{x^{**}\in X^{**}:
\forall\,A\subset X^*\ \text{with }|A|\leq\kappa,\
\exists\,x\in X\ \text{such that }x^{**}|_{A}=x|_{A}
\big\},
\]
which gives rise to the Semadeni--Pe\l czy\'nski derivative
$\mathcal{SP}_{\kappa}(X)$, see \cite{Korpalski}.

\medskip

\noindent
\textbf{The Kislyakov assignment.}
For every uncountable regular ordinal $\gamma$,
\[
\Delta_{\gamma}(X)
=
\Big\{x^{**}\in X^{**}:
\forall\,\beta<\gamma\ \text{limit},\
(f_\alpha)_{\alpha<\beta}\subset X^*\ \text{bounded},\
f_\alpha\xrightarrow{\sigma(X^*,X)}0
\Rightarrow
x^{**}(f_\alpha)\to 0
\Big\},
\]
which gives rise to the derivative used in the isomorphic classification of the spaces
$C([0,\alpha])$ for uncountable ordinals $\alpha$, see \cite{Kislyakov}.

\begin{remark}
It is straightforward to check that for every infinite cardinal $\kappa$, the Korpalski assignment is always contained in the Semadeni assignment, that is, $\Delta_\kappa(X)\subset \Delta_{\mathfrak{s}}(X)$
for every Banach space $X$. A similar inclusion holds for the Kislyakov assignment for every uncountable regular ordinal $\gamma$, namely, $\Delta_{\gamma}(X)\subset \Delta_{\mathfrak{s}}(X)$.
Since this inclusion will play an important role throughout the paper, we isolate it in the following definition.
\end{remark}

\begin{definition}\label{Def:SemadeniCondition}
A bidual assignment $\Delta$ is said to satisfy the \emph{Semadeni condition} if, for every Banach space $X$, every weak$^*$-null sequence $(x_n^*)_n$ in $X^*$, and every $\Phi\in \Delta(X)$, one has $\Phi(x_n^*)\to 0$. Equivalently, for every Banach space $X$, $\Delta(X)\subset \Delta_{\mathfrak{s}}(X)$.
\end{definition}

\begin{remark}
A Banach space $X$ is said to have the Mazur property if every weak$^*$-sequentially continuous functional in $X^{**}$ is
weak$^*$-continuous. Classical examples include separable Banach spaces, reflexive Banach spaces, and, more generally, $\operatorname{WCG}$ spaces.
This motivates the following terminology for an arbitrary bidual assignment.
\end{remark}

\begin{definition}
Let $\Delta$ be a bidual assignment. We say that a Banach space $X$
has the \emph{$\Delta$-Mazur property} (or that $X$ is a
\emph{$\Delta$-Mazur space}) if $\Delta(X)=X$.
Equivalently, $X$ is a $\Delta$-Mazur space if and only if
$\mathcal S_\Delta(X)$ is trivial.
\end{definition}

\begin{remark}
The classical Mazur property coincides with the
$\Delta_{\mathfrak s}$-Mazur property, where
$\Delta_{\mathfrak s}$ denotes the Semadeni assignment. Accordingly,
a $\Delta_{\mathfrak s}$-Mazur space will simply be called a
\emph{Mazur space}. Moreover, if a bidual assignment $\Delta$
satisfies the Semadeni condition, then every Mazur space is a
$\Delta$-Mazur space.
\end{remark}

\begin{remark}\label{Rem:Density}
For the Semadeni--Pe{\l}czy\'nski derivative associated with an
infinite cardinal $\kappa$, there is a simple sufficient condition for
a Banach space to be $\Delta_\kappa$-Mazur. Namely, if $\operatorname{dens}(B_{X^*},w^*)\leq\kappa$,
then $\mathcal{SP}_\kappa(X)=\{0\}$. In particular, this conclusion holds whenever
$\operatorname{dens}(X)\leq\kappa$.

Indeed, choose a weak$^*$-dense subset $A\subset B_{X^*}$ with
$|A|\leq\kappa$. Let $x^{**}\in\Delta_\kappa(X)$. There exists
$x_0\in X$ such that $x^{**}|_A=x_0|_A$. Now let
$x^*\in X^*$ be arbitrary. Since
$|A\cup\{x^*\}|\leq\kappa$, there exists $x_{x^*}\in X$ such that
\[
x^{**}|_{A\cup\{x^*\}}
=
x_{x^*}|_{A\cup\{x^*\}}.
\]
In particular, $x_{x^*}$ and $x_0$ coincide on $A$. Since $A$ is
weak$^*$-dense in $B_{X^*}$, it follows that $x_{x^*}=x_0$.
Consequently,
\[
x^{**}(x^*)=x_{x^*}(x^*)=x_0(x^*).
\]
Since $x^*\in X^*$ was arbitrary, we obtain
$x^{**}=x_0\in X$. Thus $\Delta_\kappa(X)=X$.

As an illustration, consider $X=\ell_\infty=\ell_1^*$. By Goldstine's
theorem, $B_{\ell_1}$ is weak$^*$-dense in
$B_{\ell_1^{**}}=B_{\ell_\infty^*}$. Since $\ell_1$ is separable, $\operatorname{dens}(B_{\ell_\infty^*},w^*)\leq\omega$. Therefore
\[
\Delta_\kappa(\ell_\infty)=\ell_\infty
\quad\text{and}\quad
\mathcal{SP}_\kappa(\ell_\infty)=\{0\}
\]
for every infinite cardinal $\kappa$.

The behavior of the Semadeni assignment on $\ell_\infty$ is quite different. Since $\ell_\infty$ has the Grothendieck property \cite[p.152]{Fabian}, every
weak$^*$-convergent sequence in $\ell_\infty^*$ is weakly convergent. Hence every element
of $\ell_\infty^{**}$ is weak$^*$-sequentially continuous, and therefore
\[
\Delta_{\mathfrak s}(\ell_\infty)=\ell_\infty^{**}
\quad\text{and}\quad
\mathcal S_{\mathfrak s}(\ell_\infty)
=
\ell_\infty^{**}/\ell_\infty.
\]
Thus, on $\ell_\infty$, the Korpalski and Semadeni assignments behave
differently even when $\kappa=\omega$.
\end{remark}

The next theorem shows that whenever $\Delta$ is a bidual assignment, the corresponding derivative is an isomorphic invariant. The proof essentially follows that of \cite[Proposition 3.1]{CandidoSquares}.

\begin{theorem}\label{Thm:DerivativeInvariant}
Let $\Delta$ be a bidual assignment, and let $X$ and $Y$ be Banach spaces. If $X$ embeds into $Y$ with distortion at most $\lambda$, then
$\mathcal S_\Delta(X)$ embeds into $\mathcal S_\Delta(Y)$ with distortion at most $2\lambda$.

If $X$ and $Y$ are linearly isomorphic with distortion at most $\lambda$, then
$\mathcal S_\Delta(X)$ and $\mathcal S_\Delta(Y)$ are linearly isomorphic with distortion at most $\lambda$.
\end{theorem}
\begin{proof}
Let $T:X\to Y$ be a linear embedding. We first observe that $T^{**}\bigl(\Delta(X)\bigr)\subset \Delta(Y)$,
which follows from the defining property of the bidual assignment.

Define $Q:\Delta(X)\to \mathcal S_\Delta(Y)$ by $Q(x^{**})=[T^{**}x^{**}]$. Then $Q$ is a well-defined bounded linear operator and
$\|Q\|\leq \|T\|$. Moreover, since $T^{**}(X)\subset Y$, we have $X\subset \ker Q$.

We claim that $\ker Q=X$. Indeed, suppose that $Q(x^{**})=0$.
Then $T^{**}x^{**}\in Y$ and hence it is weak$^*$-continuous. Since $T$ is an isomorphic
embedding, $T^*:Y^*\to X^*$ is surjective. By the Open Mapping Theorem, there exists $c>0$ such that $cB_{X^*}\subset T^*(B_{Y^*})$. Moreover, the restriction $T^*|_{B_{Y^*}}:B_{Y^*}\to T^*(B_{Y^*})$ is a weak$^*$-to-weak$^*$ continuous surjection from a compact space
onto a Hausdorff space; hence, it is a quotient map. Since $T^{**}x^{**}=x^{**}\circ T^*$ is weak$^*$-continuous on $B_{Y^*}$, it follows that $x^{**}$ is
weak$^*$-continuous on $T^*(B_{Y^*})$, and hence on $cB_{X^*}$.
By the Banach--Dieudonn\'e theorem
(see \cite[Corollary~3.94]{Fabian}), $x^{**}$ is weak$^*$-continuous
on $X^*$. Consequently, $x^{**}\in X$. We deduce that
$\ker Q\subset X$. Thus, $\ker Q=X$.

Therefore $Q$ induces an injective bounded linear operator $\widehat Q:\mathcal S_\Delta(X)\to \mathcal S_\Delta(Y)$, $\widehat Q([x^{**}])=[T^{**}x^{**}]$ with $\|\widehat Q\|\leq \|T\|$.

If $T$ is surjective, the same construction applied to $T^{-1}:Y\to X$ gives the inverse operator of $\widehat Q$, with $\|\widehat Q^{-1}\|\leq \|T^{-1}\|$. Then $\|\widehat Q\|\|\widehat Q^{-1}\| \leq \|T\|\|T^{-1}\|$.

If $T$ is not surjective, suppose that there is $[x^{**}] \in \mathcal{S}_\Delta(X)$ such that $\|[x^{**}]\|\geq \frac{3\|T^{-1}\|}{2}$ and $\|\widehat{Q}([x^{**}])\|<\frac{3}{4}$.
Let $y\in Y$ be such that $\|T^{**}(x^{**})-y\|<\frac{3}{4}$. For each $x\in X$,
\small
\begin{align*}
\frac{3}{2} &\leq \frac{1}{\|T^{-1}\|}\|x^{**}-x\|\leq \|T^{**}(x^{**})-T(x)\|\leq \|T^{**}(x^{**})-y\|+\|y-T(x)\|<\frac{3}{4}+\|y-T(x)\|.
\end{align*}
\normalsize
Taking the infimum over $x\in X$, we obtain $\operatorname{dist}(y,T(X))\geq \frac{3}{4}$. By the Hahn--Banach theorem, there is 
$y^*$ in the unit sphere of $Y^*$ such that $y^*(y) \geq \frac{3}{4}$ and $T^*(y^*)=0$. Thus
\begin{align*}
\frac{3}{4}\leq |y^*(y)|=|(T^{**}(x^{**})-y)(y^*)|\leq \|T^{**}(x^{**})-y\|<\frac{3}{4}
\end{align*}
which is a contradiction. Therefore, for each $[x^{**}]\in \mathcal{S}_\Delta(X)$, we have
\[\frac{1}{2\|T^{-1}\|}\|[x^{**}]\|\leq \|\widehat{Q}([x^{**}])\|\leq \|T\|\|[x^{**}]\|.\]

Therefore, $\widehat{Q}$ is a linear embedding with distortion $\|\widehat{Q}\|\|\widehat{Q}^{-1}\|\leq 2\|T\|\|T^{-1}\|$.
\end{proof}

\begin{remark}
Although the preceding theorem shows that the derivative behaves well
under embeddings and isomorphisms, its behavior under quotients is
more delicate. Indeed, let $T:X\to Y$ be a bounded linear surjection.
Functoriality yields a well-defined operator $\widehat T:\mathcal S_\Delta(X)\to\mathcal S_\Delta(Y)$, $\widehat T([x^{**}])=[T^{**}x^{**}]$,
but does not ensure its surjectivity. In fact, $\widehat T$ is surjective if and only if
\[
\Delta(Y)\subseteq T^{**}[\Delta(X)]+Y.
\]
Thus, preservation of quotients requires more than the functoriality
of $\Delta$. We shall return to this issue at the end of the section;
see Proposition~\ref{Prop:l1quotients}.
\end{remark}

We now investigate the behavior of bidual assignments and their associated derivatives under direct sums. We begin with $E$-direct
sums, using the notation introduced in Section~\ref{Sec:Prelim}.

\begin{theorem}\label{Thm:ell_psum_Delta}
Let $E$ be a Banach space with a normalized $1$-unconditional basis $(e_\alpha)_{\alpha<\xi}$, let $(X_\alpha)_{\alpha<\xi}$ be a family
of Banach spaces, and set $X=\big(\bigoplus_{\alpha<\xi}X_\alpha\big)_E$.
Assume that, under the canonical identification, $X^{**}=\big(\bigoplus_{\alpha<\xi}X_\alpha^{**}\big)_E$.
Then, for every bidual assignment $\Delta$, one has $\Delta(X)=
\big(\bigoplus_{\alpha<\xi}\Delta(X_\alpha)\big)_E$.
\end{theorem}
\begin{proof}
We first show that $\Delta(X)\subset \big(\bigoplus_{\alpha<\xi}\Delta(X_\alpha)\big)_E$. Let $F\in \Delta(X)$. Identifying $X^{**}=\big(\bigoplus_{\alpha<\xi}X_\alpha^{**}\big)_E$, we may write $F=(x_\alpha^{**})_\alpha$. For each $\beta<\xi$, let $P_\beta:X\to X_\beta$ denote the canonical projection. Then $P_\beta^{**}:X^{**}\to X_\beta^{**}$ satisfies $P_\beta^{**}(\Delta(X))\subset \Delta(X_\beta)$ by functoriality. Hence $x_\beta^{**}=P_\beta^{**}(F)\in \Delta(X_\beta)$, which proves the inclusion.

To prove the reverse inclusion, let $F=(x_\alpha^{**})_\alpha \in \big(\bigoplus_{\alpha<\xi}\Delta(X_\alpha)\big)_E$. For each finite subset $A\subset \xi$, define
\[
F_A=\sum_{\alpha\in A} I_\alpha^{**}(x_\alpha^{**}),
\]
where $I_\alpha:X_\alpha\to X$ is the canonical inclusion. By functoriality, $I_\alpha^{**}(\Delta(X_\alpha))\subset \Delta(X)$, hence $F_A\in \Delta(X)$. By the definition of the $E$-direct sum, the net $(F_A)_{A\in[\xi]^{<\omega}}$ converges in norm to $F$. Since $\Delta(X)$ is norm-closed, it follows that $F\in \Delta(X)$.
\end{proof}

\begin{remark}
The previous theorem shows that bidual assignments commute with arbitrary
$\ell_p$-sums, $1<p<\infty$. More generally, the result applies whenever
the bidual of an $E$-direct sum $\big(\bigoplus_{\alpha<\xi}X_\alpha\big)_E^{**}$ can be canonically identified isometrically
with the corresponding $E$-direct sum of the biduals, that is, $\big(\bigoplus_{\alpha<\xi}X_\alpha^{**}\big)_E$. Many classical sequence spaces satisfy this assumption, including suitable reflexive Orlicz and Lorentz sequence spaces.

The argument used in Theorem~\ref{Thm:ell_psum_Delta} can also be adapted to $c_0$-sums; see \cite[Lemma~3.10]{CandidoSquares}. However, additional
hypotheses are required. In particular, if $\Delta$ satisfies the Semadeni condition, then
$\Delta \big(\big(\bigoplus_{\alpha<\xi}X_\alpha\big)_{c_0}\big)= \big(\bigoplus_{\alpha<\xi}\Delta(X_\alpha)\big)_{c_0}$.
\end{remark}

For $\ell_1$-sums, the preservation of sums also depends on additional assumptions on the bidual assignment, as well 
as on the cardinality of the index set, as we show next. The arguments in the proof were inspired by \cite[Theorem 4.2]{Leung}. 
We recall that a cardinal $\mathfrak m$ is said to be
\emph{real-valued measurable} if there exists a countably additive
probability measure on $2^{\mathfrak m}$ that vanishes on singletons;
see \cite[Chapter~5]{Jech}. The existence of real-valued measurable cardinals is not provable in
$\mathrm{ZFC}$ and requires additional set-theoretic assumptions.

\begin{theorem}\label{Thm:ell_1sum_Delta}
Let $\xi$ be a cardinal that is not real-valued measurable and let $(X_\alpha)_{\alpha<\xi}$ be a family of Banach spaces. Let $\Delta$ be a bidual assignment satisfying the Semadeni condition. If $X = \big(\bigoplus_{\alpha<\xi}X_\alpha\big)_{\ell_1}$, then $\Delta(X)=\big(\bigoplus_{\alpha<\xi}\Delta(X_\alpha)\big)_{\ell_1}$.
\end{theorem}
\begin{proof}
We first prove that $\Delta(X)\subset \big(\bigoplus_{\alpha<\xi}\Delta(X_\alpha)\big)_{\ell_1}$. Let $F\in \Delta(X)$ be arbitrary. For each $\alpha<\xi$, let $I_\alpha:X_\alpha^*\to X^*$ be the canonical inclusion and define $f_\alpha=I_\alpha^*(F)\in X_\alpha^{**}$. By the functorial property of $\Delta$, we have $f_\alpha\in \Delta(X_\alpha)$. We claim that $G=(f_\alpha)_\alpha\in
\big(\bigoplus_{\alpha<\xi}\Delta(X_\alpha)\big)_{\ell_1}$.
Indeed, if $A\subset\xi$ is finite and nonempty and $\varepsilon>0$, for each $\alpha\in A$ choose
$x_\alpha^*\in B_{X_\alpha^*}$ such that
\[
\|f_\alpha\|\leq f_\alpha(x_\alpha^*)+\frac{\varepsilon}{|A|}.
\]
Putting $u=\sum_{\alpha\in A}I_\alpha(x_\alpha^*)$, we have $\|u\|\leq 1$, and hence
\[
\sum_{\alpha\in A}\|f_\alpha\|
\leq F(u)+\varepsilon
\leq \|F\|+\varepsilon.
\]
Therefore $\|G\|=\sum_{\alpha<\xi}\|f_\alpha\|\leq \|F\|$, which establishes our claim. 

Next, let $u=(x_\alpha^*)_\alpha\in \big(\bigoplus_{\alpha<\xi}X_\alpha^*\big)_{\ell_\infty}$ be arbitrary and define $m:2^\xi\to X^*$ by $m(A)=u_A$,
where
\[
u_A=
\begin{cases}
x_\alpha^*, & \alpha\in A,\\
0, & \alpha\notin A.
\end{cases}
\]
It is clear that $m$ is bounded and finitely additive. To check that $m$ is weak$^*$-continuous at the empty set, let $(A_n)_n$ be a decreasing sequence in $2^{\xi}$ such that $\bigcap_{n \in \mathbb{N}}A_n=\emptyset$. Let $x=(x_\alpha)_\alpha\in X$ be arbitrary. For each $n\in \mathbb{N}$ we have
\begin{align*}
|m(A_n)(x)|
=
\left|\sum_{\alpha\in A_n} x_{\alpha}^*(x_\alpha)\right|
\leq 
\big(\sup_{\alpha<\xi}\|x_\alpha^*\|\big)
\sum_{\alpha\in A_n}\|x_\alpha\|.
\end{align*}
Since $(A_n)_n$ decreases to the empty set and $(\|x_\alpha\|)_\alpha\in \ell_1(\xi)$, the last term tends to $0$ as $n\to\infty$. Hence $m(A_n)\xrightarrow{w^*}0$.

Now since $\Delta$ satisfies the Semadeni condition and $F\in\Delta(X)$, we have that $F(m(A_n))\to 0$. Hence the scalar set function $A\mapsto F(m(A))$
is countably additive. Therefore, as $\sum_{\alpha<\xi}\|f_\alpha\|<\infty$, the function $\mu:2^{\xi}\to \mathbb{R}$ defined by
\[
\mu(A)=(F-G)(m(A))
\]
is countably additive. Moreover, for each $\beta<\xi$ holds
\[
\mu(\{\beta\})=F(I_\beta(x_\beta^*))-f_\beta(x_\beta^*)=f_\beta(x_\beta^*)-f_\beta(x_\beta^*)=0.
\]
Thus $\mu$ is atomless. Since $\xi$ is not real-valued measurable, $\mu=0$. Hence $(F-G)(u)=0$ for every $u\in X^*$, and therefore $F=G$.

To prove the reverse inclusion, let $F=(x^{**}_\alpha)_\alpha\in \big(\bigoplus_{\alpha<\xi}\Delta(X_\alpha)\big)_{\ell_1}$,
canonically identified as a subspace of $X^{**}$. For each finite set
$A\subset \xi$, put
\[
F_A=\sum_{\alpha\in A} I_\alpha^{**}(x_\alpha^{**}).
\]
By functoriality of the bidual assignment, since $x_\alpha^{**}\in \Delta(X_\alpha)$, we have $I_\alpha^{**}(x_\alpha^{**})\in \Delta(X)$. Hence $F_A\in \Delta(X)$ for every finite $A\subset \xi$ and
\[
\|F-F_A\|
=
\sum_{\alpha\notin A}\|x_\alpha^{**}\|.
\]
It follows that the net $(F_A)_{A\in [\xi]^{<\omega}}$, where $[\xi]^{<\omega}$ is ordered by inclusion, converges in norm to $F$.
Since $\Delta(X)$ is norm closed, we conclude that $F\in \Delta(X)$.
\end{proof}

\begin{remark}
The hypothesis that $\Delta$ satisfies the Semadeni condition can be slightly weakened. Inspecting the proof of Theorem~\ref{Thm:ell_1sum_Delta}, it is enough to assume that for every Banach space $X$, every $x^{**}\in \Delta(X)$, and every bounded finitely additive set function $m:2^\Gamma\to X^*$ which is weak$^*$-continuous at the empty set, the scalar set function $\mu(A)=x^{**}(m(A))$
is countably additive. We chose to assume the Semadeni condition in the statement of
Theorem~\ref{Thm:ell_1sum_Delta} because it is a simpler and more natural
assumption on a bidual assignment, and it is consistent with the framework
developed in the next section.
\end{remark}

Having determined conditions under which bidual assignments commute with direct sums, we now pass to the corresponding derivatives. The
following elementary quotient formula allows us to do this isometrically.

\begin{lemma}\label{Lem:QuotientSum}
Let $E$ be a Banach space with a normalized $1$-unconditional basis $(e_\alpha)_{\alpha<\xi}$, let $(Y_\alpha)_{\alpha<\xi}$ be a family of Banach spaces, and for each $\alpha<\xi$ let $X_\alpha$ be a closed subspace of $Y_\alpha$. Then there is a canonical isometric identification
\[
\big(\bigoplus_{\alpha<\xi}Y_\alpha\big)_E\bigg/
\big(\bigoplus_{\alpha<\xi}X_\alpha\big)_E
\cong
\big(\bigoplus_{\alpha<\xi}Y_\alpha/X_\alpha\big)_E.
\]
\end{lemma}
\begin{proof}
Consider the canonical mapping $T:\big(\bigoplus_{\alpha<\xi}Y_\alpha\big)_E
\to\big(\bigoplus_{\alpha<\xi}Y_\alpha/X_\alpha\big)_E$ defined by
\[
T((y_\alpha)_\alpha)=([y_\alpha])_\alpha.
\]
Clearly, $T$ is well defined, $\|T\|\leq 1$, and $\ker T=\big(\bigoplus_{\alpha<\xi}X_\alpha\big)_E$.

To finish, we need to show that $T$ maps the open unit ball onto the open unit ball. Let $([y_\alpha])_\alpha$ belong to the open unit ball of $\big(\bigoplus_{\alpha<\xi}Y_\alpha/X_\alpha\big)_E$. Choose $r>1$ such that $r\|([y_\alpha])_\alpha\|<1$. For each $\alpha<\xi$ with $[y_\alpha]\neq 0$, choose $z_\alpha\in Y_\alpha$ such that $[z_\alpha]=[y_\alpha]$ and $\|z_\alpha\|<r\|[y_\alpha]\|$. If $[y_\alpha]=0$, set $z_\alpha=0$. By the $1$-unconditionality of $(e_\alpha)_{\alpha<\xi}$, $(z_\alpha)_\alpha$ belongs to $\big(\bigoplus_{\alpha<\xi}Y_\alpha\big)_E$ and
\[
\|(z_\alpha)_\alpha\|
=\left\|\sum_{\alpha<\xi}\|z_\alpha\|e_\alpha\right\|_E
\leq r\left\|\sum_{\alpha<\xi}\|[y_\alpha]\|e_\alpha\right\|_E
=r\|([y_\alpha])_\alpha\|<1.
\]
Moreover, $T((z_\alpha)_\alpha)=([y_\alpha])_\alpha$, and we are done.
\end{proof}

\begin{theorem}\label{Thm:IsometryDerivative}
Let $E$ be a Banach space with a normalized $1$-unconditional basis $(e_\alpha)_{\alpha<\xi}$, let $(X_\alpha)_{\alpha<\xi}$ be a family of Banach spaces, and set $X=\big(\bigoplus_{\alpha<\xi}X_\alpha\big)_E$. If $\Delta(X)=\big(\bigoplus_{\alpha<\xi}\Delta(X_\alpha)\big)_E$, then $\mathcal{S}_{\Delta}(X)$ is isometrically isomorphic to $\big(\bigoplus_{\alpha<\xi}\mathcal{S}_{\Delta}(X_\alpha)\big)_E$.
\end{theorem}
\begin{proof}
By Lemma~\ref{Lem:QuotientSum}, applied to $\Delta(X_\alpha)$ and $X_\alpha$, we obtain
\begin{align*}
\mathcal{S}_{\Delta}(X)
&=\Delta(X)/X
=\big(\bigoplus_{\alpha<\xi}\Delta(X_\alpha)\big)_E
\bigg/
\big(\bigoplus_{\alpha<\xi}X_\alpha\big)_E\\
&\cong
\big(\bigoplus_{\alpha<\xi}\Delta(X_\alpha)/X_\alpha\big)_E
=\big(\bigoplus_{\alpha<\xi}\mathcal{S}_{\Delta}(X_\alpha)\big)_E.
\end{align*}
This proves the theorem.
\end{proof}

We conclude this section with an observation concerning the behavior
of derivatives under quotients. If a bidual assignment satisfying the
Semadeni condition is nontrivial on a Banach space whose density is
not real-valued measurable, then the corresponding derivative fails
to preserve some quotient.

\begin{proposition}\label{Prop:l1quotients}
Let $\Delta$ be a bidual assignment satisfying the Semadeni condition. Assume that there exists a Banach space $X$ such that $\Delta(X)\neq X$
and $\operatorname{dens}(X)$ is not real-valued measurable. Then the derivative $\mathcal S_\Delta$ does not respect all quotients.
\end{proposition}
\begin{proof}
Let $\kappa=\operatorname{dens}(X)$. Since $\Delta(X)\neq X$, we have
$\mathcal S_\Delta(X)\neq\{0\}$. Moreover, there exists a bounded linear
surjection $T:\ell_1(\kappa)\twoheadrightarrow X$.
Suppose that $\mathcal S_\Delta$ respects quotients. Then the induced
operator $\widehat T:\mathcal S_\Delta(\ell_1(\kappa)) \twoheadrightarrow\mathcal S_\Delta(X)$
is surjective.

On the other hand, $\ell_1(\kappa)=\big(\bigoplus_{\alpha<\kappa}\mathbb R\big)_{\ell_1}$. Since $\Delta$ satisfies the Semadeni condition and $\kappa$ is not real-valued measurable, Theorems~\ref{Thm:ell_1sum_Delta} and \ref{Thm:IsometryDerivative} yield $\mathcal S_\Delta(\ell_1(\kappa))\cong
\big(\bigoplus_{\alpha<\kappa}\mathcal S_\Delta(\mathbb R)\big)_{\ell_1}$. Since $\mathbb R$ is reflexive, $\mathcal S_\Delta(\mathbb R)=\{0\}$,
and therefore $\mathcal S_\Delta(\ell_1(\kappa))=\{0\}$. Thus,
\[
\{0\}=\mathcal S_\Delta(\ell_1(\kappa))\twoheadrightarrow \mathcal S_\Delta(X),
\]
which is a contradiction.
\end{proof}

\begin{remark}
The preceding argument does not cover the case in which $\kappa$ is real-valued measurable. In particular, if $\kappa$ is the least
real-valued measurable cardinal, it is unclear whether the Semadeni--Pe\l czy\'nski derivative
$\mathcal{SP}_{\kappa}$ respects all quotients.
\end{remark}

\section{Derivatives and Injective tensor products}
\label{Sec:Injective}

In this section, we study the interaction between bidual assignments and
injective tensor products. We first introduce ideal-generated assignments,
which will provide the main applications of the section. We then establish
some general tensor-product results for arbitrary bidual assignments
satisfying the Semadeni condition. Finally, we return to ideal-generated
assignments and use the additional structure of testing ideals to obtain a
tensor-product formula for the associated derivatives.

The following notion isolates the structure needed to define bidual
assignments by testing bidual functionals on suitably small subsets of the
dual space.

\begin{definition}\label{Def:Test}
Let $\mathfrak I$ be a rule which assigns to each Banach space $X$
an ideal $\mathfrak I(X^*)$ of subsets of $X^*$ such that:
\begin{enumerate}
\item Every countable subset of $X^*$ belongs to $\mathfrak I(X^*)$.

\item Whenever $T:X\to Y$ is a bounded linear operator and
$A\in\mathfrak I(Y^*)$, we have
$T^*(A)\in\mathfrak I(X^*)$.
\end{enumerate}

For every Banach space $X$, define
\[
\Delta_{\mathfrak I}^0(X)=\left\{x^{**}\in X^{**}:\forall A\in\mathfrak I(X^*)\ \exists x\in X\text{ such that }x^{**}|_A=x|_A\right\}
\]
and $\Delta_{\mathfrak I}(X)=\overline{\Delta_{\mathfrak I}^0(X)}^{\|\cdot\|}$.

We say that $\mathfrak I$ is a \emph{testing ideal} if, in addition,
the following condition holds:
\begin{enumerate}
\setcounter{enumi}{2}
\item Whenever $X$ and $Y$ are Banach spaces,
$A\in\mathfrak I(\mathcal L(X,Y^*))$ and
$y^{**}\in\Delta_{\mathfrak I}(Y)$, then $\{y^{**}\circ T:T\in A\}\in\mathfrak I(X^*)$.
\end{enumerate}

The associated derivative of $X$ is denoted by $\mathcal S_{\mathfrak I}(X)=\Delta_{\mathfrak I}(X)/X$.
\end{definition}

\begin{remark}\label{Rem:ExamplesAssignments}
It is straightforward to verify that, for every infinite cardinal $\kappa$, the testing ideal
\[\mathfrak I_\kappa(X^*)=\{A\subset X^*:|A|\leq\kappa\}\]
satisfies Definition~\ref{Def:Test}. Moreover, for every uncountable cardinal $\kappa$, the rule
\[\mathfrak I_{<\kappa}(X^*)=\{A\subset X^*:|A|<\kappa\}\]
also defines a testing ideal.

The bidual assignment associated with $\mathfrak I_\kappa$
recovers the Korpalski assignment $\Delta_\kappa$ presented in the
previous section.
\end{remark}

\begin{proposition}\label{Prop:Bidual}
Let $\mathfrak I$ be a testing ideal. The ideal-generated assignment $\Delta_{\mathfrak I}$ is a bidual assignment in the sense of Definition \ref{Def:BidualAssignment}. Moreover,
$\Delta_{\mathfrak I}$ satisfies the Semadeni condition (see Definition \ref{Def:SemadeniCondition})
\end{proposition}
\begin{proof}
We first check the functoriality. Let $T:X\to Y$ be a bounded linear operator, and let 
$x^{**}\in \Delta_{\mathfrak I}^0(X)$, and $A\in \mathfrak I(Y^*)$ be arbitrary. Since $\mathfrak I$ is a
testing ideal, we have $T^*(A)\in \mathfrak I(X^*)$.
Since $x^{**}\in \Delta_{\mathfrak I}^0(X)$, there exists
$x\in X$ such that $x^{**}|_{T^*(A)}=x|_{T^*(A)}$.

 Now, if $y^*\in A$, then
\[T^{**}(x^{**})(y^*)=x^{**}(T^*y^*)=x(T^*y^*)=y^*(Tx).\]
Therefore, $T^{**}(x^{**})|_A=(Tx)|_A$. Thus  $T^{**}(\Delta_{\mathfrak I}^0(X))\subset \Delta_{\mathfrak I}^0(Y)$. Now let $x^{**}\in \Delta_{\mathfrak I}(X)$.
By definition, there exists a sequence $(x_n^{**})_n$ in $\Delta_{\mathfrak I}^0(X)$
converging in norm to $x^{**}$. Since $T^{**}$ is norm continuous,
$(T^{**}(x_n^{**}))_n$ converges in norm to $T^{**}(x^{**})$.
Since $T^{**}(x_n^{**})\in \Delta_{\mathfrak I}^0(Y)$ for every $n$, we conclude that $T^{**}(x^{**})\in
\overline{\Delta_{\mathfrak I}^0(Y)}=\Delta_{\mathfrak I}(Y)$.

Finally, let $x^{**}\in \Delta_{\mathfrak I}^0(X)$ be arbitrary and let $(x_n^*)_n$ be a weak$^*$-null sequence in $X^*$. Since the set $A=\{x_n^*:n \in \mathbb{N}\}$ is countable, $A\in \mathfrak{I}(X^*)$, and there is $x\in X$ such that $x^{**}|_{A}=x|_{A}$. Therefore
\[
\lim_{n\to \infty}x^{**}(x_n^*)
=
\lim_{n\to \infty}x_n^*(x)
=
0.
\]
Thus $\Delta_{\mathfrak I}^0(X)\subset \Delta_{\mathfrak s}(X)$ where $\Delta_{\mathfrak s}$ denotes the Semadeni assignment introduced in the previous section. Since $\Delta_{\mathfrak s}(X)$ is norm closed, it follows that $\Delta_{\mathfrak I}(X)=\overline{\Delta_{\mathfrak I}^0(X)}\subset\Delta_{\mathfrak s}(X)$.
Hence $\Delta_{\mathfrak I}$ satisfies the Semadeni condition.
\end{proof}

Before specializing to ideal-generated assignments, we establish a
general tensor-product result for assignments satisfying the Semadeni
condition. More precisely, we seek conditions ensuring that
\[
\Delta(X\widehat{\otimes}_\varepsilon Y)
\cong
X\widehat{\otimes}_\varepsilon\Delta(Y),
\]
canonically and isometrically. This will yield, under a natural
complementability assumption, the corresponding relation
\[
\mathcal S_\Delta(X\widehat{\otimes}_\varepsilon Y)
\sim
X\widehat{\otimes}_\varepsilon\mathcal S_\Delta(Y).
\]

Related questions have been studied in \cite{CandidoSquares} in the context of the Semadeni derivative, and 
\cite{galego1,galego2,galego3,galego4}
in the context of the Kislyakov derivative and Banach spaces of the form
$C([0,\alpha])$, where $\alpha$ is an uncountable ordinal.

We begin by introducing some additional notation and recalling a standard identification. For Banach spaces $X$ and $Y$, we denote by $\mathcal K_{w^*}(X^*,Y)$ the space of compact operators from $X^*$ to $Y$ that are weak$^*$-to-weak continuous. Every element of $X\widehat{\otimes}_{\varepsilon}Y$ can be canonically identified with a compact weak$^*$-to-weak continuous operator from $X^*$ into $Y$, where an
elementary tensor $x\otimes y$ corresponds to the operator
\[
(x\otimes y)(x^*)=x^*(x)y,
\]
yielding an isometric embedding $X\widehat{\otimes}_{\varepsilon}Y
\hookrightarrow \mathcal K_{w^*}(X^*,Y)$. If either $X$ or $Y$ has the approximation property, then this embedding is surjective and hence an isometric isomorphism.

We shall also need the following canonical construction. Consider the bilinear map $G:X^*\times Y^*\to (X\widehat{\otimes}_{\varepsilon}Y)^*$
defined on elementary tensors by the formula
\[
G(x^*,y^*)\big(\sum_{i=1}^n a_i\,(x_i\otimes y_i)\big)=
\sum_{i=1}^n a_i x^*(x_i)y^*(y_i).
\]
Since $\|G(x^*,y^*)\|\leq \|x^*\|\,\|y^*\|$, the universal property of the projective tensor product \cite[Theorem 2.9]{Ryan} yields a canonical mapping $\widehat G:X^*\widehat{\otimes}_{\pi}Y^*\to (X\widehat{\otimes}_{\varepsilon}Y)^*$ with $\|\widehat G\|\leq 1$. 

The following lemma will play an important role in the subsequent results.
It associates with each element of the bidual two slice operators and shows
that these operators respect the bidual assignment.

\begin{lemma}\label{Lem:SliceOperators}
Let $X$ and $Y$ be Banach spaces, let $\Delta$ be a bidual assignment, and set
$E=X\widehat{\otimes}_\varepsilon Y$. For each $\Phi\in E^{**}$, define
$L_\Phi:X^*\to Y^{**}$ and $R_\Phi:Y^*\to X^{**}$ by
\[
L_\Phi(x^*)(y^*)=\Phi\bigl(\widehat G(x^*\otimes y^*)\bigr)
\quad\text{and}\quad
R_\Phi(y^*)(x^*)=\Phi\bigl(\widehat G(x^*\otimes y^*)\bigr).
\]
Then the assignments $\Phi\mapsto L_\Phi$ and $\Phi\mapsto R_\Phi$ are
linear and contractive. Moreover, $L_u=u$ and $R_u=u$ for every $u\in E$
under the canonical identifications.

Furthermore, if $\Phi\in\Delta(E)$, then $L_\Phi(X^*)\subset\Delta(Y)$ and $R_\Phi(Y^*)\subset\Delta(X)$.
\end{lemma}
\begin{proof}
For every $x^*\in X^*$ and $y^*\in Y^*$, we have
\[
\big|L_\Phi(x^*)(y^*)\big|
=
\big|R_\Phi(y^*)(x^*)\big|
\leq
\|\Phi\|\,\|x^*\|\,\|y^*\|.
\]
Hence $L_\Phi$ and $R_\Phi$ are bounded linear operators satisfying $\|L_\Phi\|\leq\|\Phi\|$ and $\|R_\Phi\|\leq\|\Phi\|$.
It is also clear from their definitions that the assignments $\Phi\mapsto L_\Phi$ and $\Phi\mapsto R_\Phi$ are linear and contractive.

Let $u\in E$. For every $x^*\in X^*$ and $y^*\in Y^*$, we have
\[
L_u(x^*)(y^*)=\widehat G(x^*\otimes y^*)(u)=y^*(u(x^*)).
\]
Therefore, $L_u=u$ under the canonical identification of $E$ with a
subspace of $\mathcal K_{w^*}(X^*,Y)$. Similarly, $R_u=u$ under the
canonical identification of $E$ with a subspace of
$\mathcal K_{w^*}(Y^*,X)$.

Now suppose that $\Phi\in\Delta(E)$. For each $x^*\in X^*$, consider the
bounded linear operator $P_{x^*}:E\to Y$ defined by
\[
P_{x^*}(u)=u(x^*).
\]
Then $P_{x^*}^*(y^*)=\widehat G(x^*\otimes y^*)$ for every $y^*\in Y^*$. By the functoriality of $\Delta$, we have
$P_{x^*}^{**}(\Phi)\in\Delta(Y)$. Moreover,
\[
P_{x^*}^{**}(\Phi)(y^*)=\Phi(P_{x^*}^*(y^*))=\Phi\bigl(\widehat G(x^*\otimes y^*)\bigr)
=
L_\Phi(x^*)(y^*).
\]
Therefore, $L_\Phi(x^*)=P_{x^*}^{**}(\Phi)\in\Delta(Y)$, and hence $L_\Phi(X^*)\subset\Delta(Y)$.

Similarly, for each $y^*\in Y^*$, consider the bounded linear operator
$Q_{y^*}:E\to X$ defined by
\[
Q_{y^*}(u)=u(y^*),
\]
where $E$ is canonically identified with a subspace of
$\mathcal K_{w^*}(Y^*,X)$. The same argument gives $R_\Phi(y^*)=Q_{y^*}^{**}(\Phi)\in\Delta(X)$. Consequently, $R_\Phi(Y^*)\subset\Delta(X)$.
\end{proof}

The proof of the next theorem is motivated by \cite[Theorem 2.2]{Leung} and \cite[Proposition 5.1]{Kapp}. 

\begin{theorem}\label{Thm:DeltaInjectiveTensor}
Let $X$ and $Y$ be Banach spaces, and let $\Delta$ be a bidual assignment
satisfying the Semadeni condition. Assume that
\begin{enumerate}[label=\textnormal{(\alph*)}]
\item \label{it:aaaa} $\Delta(X)=X$;
\item \label{it:bbbb} $\widehat{G}(X^*\otimes Y^*)$ is weak$^*$-sequentially dense in
$(X\widehat{\otimes}_\varepsilon Y)^*$;
\item \label{it:cccc} $X$ has the Gelfand--Phillips property;
\item \label{it:dddd} either $X$ or $\Delta(Y)$ has the approximation property.
\end{enumerate}
Then the map $T:\Delta(X\widehat{\otimes}_\varepsilon Y)\to X\widehat{\otimes}_\varepsilon\Delta(Y)$ defined by $T(\Phi)=L_\Phi$, where
\[
L_\Phi(x^*)(y^*)=\Phi\bigl(\widehat G(x^*\otimes y^*)\bigr),
\]
is a linear isometric isomorphism. Moreover, $T(u)=u$ for every $u\in X\widehat{\otimes}_\varepsilon Y$,
under the canonical identification of $X\widehat{\otimes}_\varepsilon Y$ with a subspace of $X\widehat{\otimes}_\varepsilon\Delta(Y)$.
\end{theorem}
\begin{proof}
For the benefit of the reader, we divide the proof into several steps. Set
$E=X\widehat{\otimes}_\varepsilon Y$.

\medskip
\noindent\emph{\textbf{Step 1}. Well-definedness and basic properties of $T$.}

Fix $\Phi\in\Delta(E)$. By Lemma~\ref{Lem:SliceOperators} and assumption~\ref{it:aaaa}
\[
L_\Phi(X^*)\subset\Delta(Y)
\quad\text{and}\quad
R_\Phi(Y^*)\subset\Delta(X)=X.
\]
We first prove that $R_\Phi:Y^*\to X$ is compact. We begin with a useful
observation. Let $(x_n^*)$ be a weak$^*$-null sequence in $X^*$ and let
$(y_n^*)$ be a bounded sequence in $Y^*$. For every $u\in E$, regarded
canonically as an operator from $X^*$ into $Y$, we have
\[
\widehat G(x_n^*\otimes y_n^*)(u)
=
y_n^*(u(x_n^*))
\to0,
\]
because $u:X^*\to Y$ is weak$^*$-to-norm continuous. Hence
$(\widehat G(x_n^*\otimes y_n^*))_n$ is weak$^*$-null in $E^*$.
 
By assumption~\ref{it:cccc}, it is enough to show that $R_\Phi(B_{Y^*})$ is
limited in $X$. Let $(x_n^*)$ be a weak$^*$-null sequence in $X^*$. For each $n$, choose $y_n^*\in B_{Y^*}$ such that
\[
\sup_{y^*\in B_{Y^*}}|x_n^*(R_\Phi(y^*))|
\leq |x_n^*(R_\Phi(y_n^*))|+\frac1n.
\]
From the observation above, $(\widehat{G}(x_n^*\otimes y_n^*))_n$ is weak$^*$-null in $E^*$. Since $\Delta$ satisfies the Semadeni condition, we obtain that 
\[
\lim_{n\to \infty}x_n^*(R_\Phi(y_n^*))=\Phi(\widehat{G}(x_n^*\otimes y_n^*))=0.
\]
It follows that
\[
\lim_{n\to \infty}\sup_{y^*\in B_{Y^*}}|x_n^*(R_\Phi(y^*))|=0.
\]
Thus $R_\Phi(B_{Y^*})$ is limited. By the Gelfand--Phillips property of $X$, we deduce that $R_\Phi(B_{Y^*})$ is
relatively norm compact. Therefore $R_\Phi$ is a compact operator.

For every $x^*\in X^*$ and $y^*\in Y^*$,
\[
R_\Phi^*(x^*)(y^*)
=
x^*(R_\Phi(y^*))
=
L_\Phi(x^*)(y^*).
\]
 Thus $L_\Phi=R_\Phi^*$. By Schauder's theorem \cite[Theorem 15.3]{Fabian}, $L_\Phi$ is compact. Since $R_\Phi$ is compact and takes values in $X$, the range of $R_\Phi^{**}$ is contained in $X$. Hence $L_\Phi$ is weak$^*$-to-weak continuous. Therefore $L_\Phi\in\mathcal K_{w^*}(X^*,\Delta(Y))$. By assumption~\ref{it:dddd}, we may canonically identify this space isometrically with $X\widehat{\otimes}_\varepsilon\Delta(Y)$. Thus $T(\Phi)=L_\Phi$
belongs to the stated codomain.
 
Therefore, $T$ is well defined, linear, and contractive. Moreover, by Lemma~\ref{Lem:SliceOperators}, $T(u)=u$ for every $u\in E$, under the
canonical identification of $E$ with a subspace of $X\widehat{\otimes}_\varepsilon\Delta(Y)$.
 
\medskip
\noindent\emph{\textbf{Step 2}. Injectivity of $T$.}

Suppose that $\Phi\in\Delta(E)$ and $T(\Phi)=0$. Then
$\Phi(\widehat G(x^*\otimes y^*))=0$ for all $x^*\in X^*$ and
$y^*\in Y^*$. Since $\Phi$ is weak$^*$-sequentially continuous,
$\ker\Phi$ is weak$^*$-sequentially closed in $E^*$. By
assumption~\ref{it:bbbb}, this kernel is $E^*$, and hence $\Phi=0$.

\medskip
\noindent\emph{\textbf{Step 3}. Isometric lifting of the algebraic tensor product.}

Let $v=\sum_{j=1}^m x_j\otimes y_j^{**}\in X\otimes \Delta(Y)$ be arbitrary. For each
$j$, define $I_j:Y\to E$ by $I_j(y)=x_j\otimes y$. Set
\[
\Phi=\sum_{j=1}^m y_j^{**}\circ I_j^*
=
\sum_{j=1}^m I_j^{**}(y_j^{**})
\in E^{**}.
\]
Since each $y_j^{**}\in \Delta(Y)$, it follows by functoriality that
$I_j^{**}(y_j^{**})\in \Delta(E)$. Hence $\Phi\in \Delta(E)$. By noting that 
$I_j^*(\widehat{G}(x^*\otimes y^*))=x^*(x_j)y^*$ for each $1\leq j \leq m$, we deduce that $T(\Phi)=v$. 

We next show that $\Phi$ and $v$ have the same norm. Let $\psi\in E^*$ and $\varepsilon>0$ be arbitrary. Let $F=\operatorname{span}\{y_j^{**}:1\leq j\leq m\}\subset \Delta(Y)\subset Y^{**}$. By the principle of local reflexivity, there exists a linear operator $U:F\to Y$ such that $\|U\|\leq 1+\varepsilon$ and $I_j^*(\psi)(U(y_j^{**}))=y_j^{**}(I_j^*(\psi))$ for each $1\leq j\leq m$. Consequently,
\begin{align*}
|\Phi(\psi)|&=\left|\sum_{j=1}^m y_j^{**}(I_j^*(\psi))\right|
=\left|\sum_{j=1}^m I_j^*(\psi)(U(y_j^{**}))\right| \\
&=\left|\sum_{j=1}^m \psi(x_j\otimes U(y_j^{**}))\right|
\leq \|\psi\|\,
\left\|\sum_{j=1}^m x_j\otimes U(y_j^{**})\right\|
\leq (1+\varepsilon)\|\psi\|\,\|v\|.
\end{align*}
Since $\varepsilon>0$ is arbitrary, we get $\|\Phi\|\leq \|v\|$. As
$\|T\|\leq 1$ and $T(\Phi)=v$, it follows that $\|\Phi\|=\|v\|$.

\medskip
\noindent\emph{\textbf{Step 4}. Extension and conclusion.}

Since $T$ is injective, Step 3 defines a linear isometry
$S_0:X\otimes\Delta(Y)\to\Delta(E)$ by assigning to each $v$ its unique
preimage $\Phi$. Moreover, $T\circ S_0$ is the identity on
$X\otimes\Delta(Y)$. Since $X\otimes\Delta(Y)$ is dense in
$X\widehat{\otimes}_\varepsilon\Delta(Y)$ and $\Delta(E)$ is complete,
$S_0$ extends uniquely to a linear isometry
$S:X\widehat{\otimes}_\varepsilon\Delta(Y)\to\Delta(E)$. By continuity,
\[
T\circ S=\operatorname{Id}_{X\widehat{\otimes}_\varepsilon\Delta(Y)}.
\]
Thus $T$ is surjective. Since $T$ is also injective, $S=T^{-1}$, and
therefore $T$ is a linear isometric isomorphism.
\end{proof}

We can now address the question raised at the beginning of this section. The canonical isometry obtained above yields the desired relation between derivatives when $Y$ is complemented in $\Delta(Y)$.

\begin{theorem}\label{Thm:TensorQuotientDerivative}
Let $X$ and $Y$ be Banach spaces, and let $\Delta$ be a bidual assignment
satisfying the Semadeni condition. Assume that the hypotheses of
Theorem~\ref{Thm:DeltaInjectiveTensor} hold and that $Y$ is complemented
in $\Delta(Y)$ (as is always the case when $Y$ is a dual Banach space).
Then
\[
\mathcal S_\Delta(X\widehat{\otimes}_\varepsilon Y)
\sim
X\widehat{\otimes}_\varepsilon\mathcal S_\Delta(Y).
\]
More precisely, if $P:\Delta(Y)\to Y$ is a bounded projection, then
the isomorphism above can be chosen with distortion at most $\|\operatorname{Id}_{\Delta(Y)}-P\|^2$.
\end{theorem}
\begin{proof}
By Theorem~\ref{Thm:DeltaInjectiveTensor}, $T:\Delta(X\widehat{\otimes}_\varepsilon Y)\to
X\widehat{\otimes}_\varepsilon\Delta(Y)$ is a surjective isometry and
$T(u)=u$ for every $u\in X\widehat{\otimes}_\varepsilon Y$. Therefore, $T$ induces a surjective
isometry $\widetilde T:\mathcal{S}_{\Delta}(X\widehat{\otimes}_\varepsilon Y)\to 
\frac{X\widehat{\otimes}_\varepsilon \Delta(Y)}
     {X\widehat{\otimes}_\varepsilon Y}$, with $\widetilde T([\varphi])=[T(\varphi)]$.

Let $P:\Delta(Y)\to Y$ be a bounded projection and let $W=\ker P$.
Then $\Delta(Y)=Y\oplus W$ and
\[
W\sim\Delta(Y)/Y=\mathcal S_\Delta(Y),
\]
where the isomorphism has distortion at most
$\|\operatorname{Id}_{\Delta(Y)}-P\|$. This induces an isomorphism
\[
X\widehat{\otimes}_\varepsilon W
\sim
X\widehat{\otimes}_\varepsilon\mathcal S_\Delta(Y)
\]
with the same distortion; see \cite[Proposition~3.2]{Ryan}.

The operator $\operatorname{Id}_X\otimes P:X\widehat{\otimes}_\varepsilon\Delta(Y)\to X\widehat{\otimes}_\varepsilon Y$
is a bounded projection; see \cite[Proposition~3.2]{Ryan}. Hence $X\widehat{\otimes}_\varepsilon Y$ is complemented in
$X\widehat{\otimes}_\varepsilon\Delta(Y)$. Moreover, $\ker(\operatorname{Id}_X\otimes P)=X\widehat{\otimes}_\varepsilon W$.
Consequently,
\[
\frac{X\widehat{\otimes}_\varepsilon\Delta(Y)}
     {X\widehat{\otimes}_\varepsilon Y}
\sim
X\widehat{\otimes}_\varepsilon W,
\]
where the isomorphism has distortion at most $\|\operatorname{Id}_{\Delta(Y)}-P\|$.

Combining the estimates above, we obtain
\[
\frac{X\widehat{\otimes}_\varepsilon\Delta(Y)}
     {X\widehat{\otimes}_\varepsilon Y}
\sim
X\widehat{\otimes}_\varepsilon\mathcal S_\Delta(Y),
\]
with distortion at most $\|\operatorname{Id}_{\Delta(Y)}-P\|^2$. This completes the proof.
\end{proof}

For a $C(K,Y)$ space, the conclusion of Theorem~\ref{Thm:TensorQuotientDerivative} can be strengthened:
no complementation assumption is needed, and the resulting
identification is isometric. The key point is that injective tensor
products with $C(K)$ preserve quotient operators.

\begin{theorem}\label{Thm:CKTensorQuotient}
Let $K$ be a compact Hausdorff space, let $Y$ be a Banach space, and
let $\Delta$ be a bidual assignment satisfying the Semadeni condition.
Assume that the hypotheses of
Theorem~\ref{Thm:DeltaInjectiveTensor} hold with $X=C(K)$. Then
\[
\mathcal S_\Delta
\bigl(C(K,Y)\bigr)
\cong
C\bigl(K,\mathcal S_\Delta(Y)\bigr)
\]
isometrically.
\end{theorem}
\begin{proof}
As in the proof of
Theorem~\ref{Thm:TensorQuotientDerivative},
Theorem~\ref{Thm:DeltaInjectiveTensor} gives a canonical surjective
linear isometry
\[
\mathcal S_\Delta
\bigl(C(K)\widehat\otimes_\varepsilon Y\bigr)
\to
\frac{C(K)\widehat\otimes_\varepsilon\Delta(Y)}
     {C(K)\widehat\otimes_\varepsilon Y}.
\]

Let $q:\Delta(Y)\to\mathcal S_\Delta(Y)$ be the canonical quotient
map. By \cite[Proposition~3.5]{Ryan},
\[
\operatorname{Id}_{C(K)}\otimes q:
C(K)\widehat\otimes_\varepsilon\Delta(Y)
\to
C(K)\widehat\otimes_\varepsilon\mathcal S_\Delta(Y)
\]
is a quotient operator. Under the canonical identification with
spaces of continuous vector-valued functions, this operator is given
by $f\mapsto q\circ f$. Hence $\ker(\operatorname{Id}_{C(K)}\otimes q)= C(K,Y) = C(K)\widehat\otimes_\varepsilon Y$.
Consequently,
\[
\frac{C(K)\widehat\otimes_\varepsilon\Delta(Y)}
     {C(K)\widehat\otimes_\varepsilon Y}
\cong
C(K)\widehat\otimes_\varepsilon\mathcal S_\Delta(Y)
\]
isometrically. Combining this with the preceding isometry and the
canonical identification
\[
C(K)\widehat\otimes_\varepsilon\mathcal S_\Delta(Y)
\cong
C\bigl(K,\mathcal S_\Delta(Y)\bigr)
\]
completes the proof.
\end{proof}

We now return to ideal-generated assignments. Our next goal is to obtain a
symmetric identification involving both $\Delta_{\mathfrak I}(X)$ and
$\Delta_{\mathfrak I}(Y)$. We first record the following standard
intersection property of the injective tensor product.

\begin{lemma}\label{Lem:Intersection}
Let $X$ and $Y$ be Banach spaces, and let $A\subset X$ and $B\subset Y$ be
closed subspaces. If $A$ has the approximation property, then, inside
$X\widehat{\otimes}_{\varepsilon}Y$, one has
\[
\bigl(X\widehat{\otimes}_{\varepsilon}B\bigr)
\cap\bigl(A\widehat{\otimes}_{\varepsilon}Y\bigr)=A\widehat{\otimes}_{\varepsilon}B .
\]
\end{lemma}
\begin{proof}
The inclusion $A\widehat{\otimes}_{\varepsilon}B\subset\bigl(X\widehat{\otimes}_{\varepsilon}B\bigr)\cap\bigl(A\widehat{\otimes}_{\varepsilon}Y\bigr)
$ is immediate. Conversely, let $u\in \bigl(X\widehat{\otimes}_{\varepsilon}B\bigr)\cap \bigl(A\widehat{\otimes}_{\varepsilon}Y\bigr)$ be arbitrary. We regard $u$ as an element of $A\widehat{\otimes}_{\varepsilon}Y$. Since $A$ has the approximation property, there exists a net
$(S_\alpha)_{\alpha\in I}$ of finite-rank operators on $A$ such that $\lim_\alpha(S_\alpha\otimes I_Y)v=v$  for every $v\in A\widehat{\otimes}_{\varepsilon}Y$. For a given $\alpha\in I$, assume that $S_\alpha:A\to A$ is given by the formula
\[S_\alpha(x)=\sum_{j=1}^{m} a_j a_j^*(x).\]
with $a_j\in A$ and $a_j^*\in A^*$. By Hahn--Banach theorem, choose norm-preserving extensions
$\widetilde a_j^*\in X^*$ of $a_j^*$. Since $u\in X\widehat{\otimes}_{\varepsilon}B$, it follows that $(\widetilde a_j^*\otimes I_Y)u\in B$, for each $j=1,\dots,m$. Therefore
\[
(S_\alpha\otimes I_Y)u=\sum_{j=1}^{m}
a_j\otimes(\widetilde a_j^*\otimes I_Y)u
\in A\otimes B.
\]
Hence $(S_\alpha\otimes I_Y)u\in A\widehat{\otimes}_{\varepsilon}B$ for every $\alpha$. Recalling that  $(S_\alpha\otimes I_Y)u\to u$,
we conclude that $u\in A\widehat{\otimes}_{\varepsilon}B$. This proves the reverse inclusion.
\end{proof}

The following lemma is a key ingredient in the proof of the symmetric
tensor-product formula. Its argument is inspired by \cite[Proposition~5.1]{Kapp}.

\begin{theorem}\label{Thm:TensorEmbeddingIntersection}
Let $X$ and $Y$ be Banach spaces. Let $\mathfrak I$ be a testing ideal and $\Delta_{\mathfrak I}$ the associated ideal-generated assignment. Assume that
\begin{itemize}
    \item[(a)] $X^{**}$ and $Y^{**}$ have the approximation property;
    \item[(b)]
$\widehat{G}(B_{X^*\widehat{\otimes}_\pi Y^*})$ is weak$^*$-sequentially dense in $B_{(X\widehat{\otimes}_\varepsilon Y)^*}$;
    \item[(c)] $B_{X^*}$ and $B_{Y^*}$ are weak$^*$-sequentially compact;
    \item[(d)] either $\Delta_{\mathfrak I}(X)$ or $\Delta_{\mathfrak I}(Y)$ has the approximation property.
\end{itemize}
Then there is a canonical isometric isomorphism $J:\Delta_{\mathfrak I}(X\widehat{\otimes}_{\varepsilon}Y) \to \Delta_{\mathfrak I}(X) \widehat{\otimes}_{\varepsilon} \Delta_{\mathfrak I}(Y)$ such that $J(u)=u$ for every $u\in X\widehat{\otimes}_{\varepsilon}Y$.
\end{theorem}
\begin{proof}
The proof is divided into four steps. The first three use only the
general properties of bidual assignments satisfying the Semadeni
condition, whereas the final step uses the additional structure of the
testing ideal to establish surjectivity.

\medskip
\noindent\emph{\textbf{Step 1}. Compactness of the slice operators.}

Set $E=X\widehat{\otimes}_\varepsilon Y$ and fix $\Phi\in\Delta_{\mathfrak I}(E)$. By Lemma~\ref{Lem:SliceOperators},
$L_\Phi(X^*)\subset\Delta_{\mathfrak I}(Y)$ and $R_\Phi(Y^*)\subset\Delta_{\mathfrak I}(X)$. We prove that $L_\Phi$ is a compact operator. Indeed, let $(x_n^*)_n$ be a sequence in $B_{X^*}$. Since $B_{X^*}$ is weak$^*$-sequentially compact, after
passing to a subsequence we may assume that $x_n^*\xrightarrow{w^*}x^*$ for some $x^*\in B_{X^*}$. We claim that $\|L_\Phi(x_n^*)-L_\Phi(x^*)\|\to 0$.
Otherwise, passing to a further subsequence, there exist $\varepsilon>0$ and $y_n^*\in B_{Y^*}$ such that
\[\left|\Phi(\widehat{G}((x_n^*-x^*)\otimes y_n^*))\right|=\left|\bigl(L_\Phi(x_n^*)-L_\Phi(x^*)\bigr)(y_n^*)\right|\geq \varepsilon\]
for every $n$. Since $(x_n^*-x^*)_n$ is weak$^*$-null in $X^*$ and $(y_n^*)_n$ is bounded, the observation in the proof of
Theorem~\ref{Thm:DeltaInjectiveTensor} shows that $(\widehat G((x_n^*-x^*)\otimes y_n^*))_n$ is weak$^*$-null in $E^*$. Since $\Phi\in \Delta_{\mathfrak I}(E)$ and $\Delta_{\mathfrak I}$ satisfies the Semadeni condition, we obtain that $\Phi(\widehat{G}((x_n^*-x^*)\otimes y_n^*))\to 0$, a contradiction. Therefore $L_\Phi(B_{X^*})$ is relatively norm compact, and hence $L_\Phi$ is a compact operator. Similarly, using the weak$^*$-sequential compactness of $B_{Y^*}$, we also conclude that $R_\Phi$ is a compact operator.

\medskip
\noindent\emph{\textbf{Step 2}. Construction of $J$.}

Since $X^{**}$ has the approximation property, the compact operator
$L_\Phi:X^*\to \Delta_{\mathfrak I}(Y)$ corresponds canonically to an element of
$X^{**}\widehat{\otimes}_{\varepsilon}\Delta_{\mathfrak I}(Y)$, see
\cite[Corollary 4.13]{Ryan}. Similarly, since $Y^{**}$ has the approximation property, the compact
operator $R_\Phi:Y^*\to \Delta_{\mathfrak I}(X)$ corresponds canonically to an
element of $\Delta_{\mathfrak I}(X)\widehat{\otimes}_{\varepsilon}Y^{**}$. Both elements induce the same bilinear form
\[
(x^*,y^*)\mapsto\Phi(\widehat G(x^*\otimes y^*)).
\]
Thus we define $J(\Phi)=L_\Phi$, viewed as an element of the intersection $\bigl(X^{**}\widehat{\otimes}_{\varepsilon}\Delta_{\mathfrak I}(Y) \bigr)
\cap
\bigl(\Delta_{\mathfrak I}(X) \widehat{\otimes}_{\varepsilon}Y^{**}\bigr)$. And since either $\Delta_{\mathfrak I}(X)$ or $\Delta_{\mathfrak I}(Y)$ has the approximation property, it follows by Lemma \ref{Lem:Intersection} that $J$ is a well defined mapping from $\Delta_{\mathfrak I}(X\widehat{\otimes}_{\varepsilon}Y)$ to  $\Delta_{\mathfrak I}(X) \widehat{\otimes}_{\varepsilon} \Delta_{\mathfrak I}(Y)$. 

\medskip
\noindent\emph{\textbf{Step 3}. Isometry and compatibility with the original tensor product.}

It is readily seen that $J$ is linear and that $\|J(\Phi)\|\leq \|\Phi\|$ for every $\Phi \in \Delta_{\mathfrak I}(E)$. To prove the reverse inequality, set
$C_\Phi=\{\psi\in B_{E^*}:|\Phi(\psi)|\leq\|J(\Phi)\|\}$. Since $\Phi$ is weak$^*$-sequentially continuous, $C_\Phi$ is weak$^*$-sequentially closed in $B_{E^*}$. Moreover, the unit ball of
$X^*\widehat{\otimes}_{\pi}Y^*$ is the closed absolutely convex hull of the elementary tensors $x^*\otimes y^*$ with $\|x^*\|\leq1$ and
$\|y^*\|\leq1$. Hence, for every $\eta\in\widehat G(B_{X^*\widehat{\otimes}_{\pi}Y^*})$,
\[
|\Phi(\eta)|\leq\sup_{\|x^*\|\leq1,\ \|y^*\|\leq1}|\Phi(\widehat G(x^*\otimes y^*))|=\sup_{\|x^*\|\leq 1}\|L_\Phi(x^*)\|=\|L_\Phi\|=\|J(\Phi)\|.
\]

Thus $\widehat G(B_{X^*\widehat{\otimes}_{\pi}Y^*})\subset C_\Phi$. By assumption~(b), $C_\Phi=B_{E^*}$, and therefore
$\|\Phi\|\leq\|J(\Phi)\|$. Thus $J$ is an isometry.

By Lemma~\ref{Lem:SliceOperators}, $L_u=u$ under the canonical operator identification for every $u\in E$. Since $J(u)=L_u$, the
canonical operator identification used in Step 2 gives $J(u)=u$ in $\Delta_{\mathfrak I}(X)\widehat{\otimes}_\varepsilon
\Delta_{\mathfrak I}(Y)$.

\medskip
\noindent\emph{\textbf{Step 4}. Surjectivity.}

Since $J$ is an isometry, its range is closed. Hence it is enough to show that the range
contains the dense subspace $\Delta_{\mathfrak I}^0(X)\otimes \Delta_{\mathfrak I}^0(Y)$ of $\Delta_{\mathfrak I}(X)\widehat{\otimes}_{\varepsilon}
\Delta_{\mathfrak I}(Y)$. Let $a^{**}\in \Delta_{\mathfrak{I}}^0(X)$ and $b^{**}\in \Delta_{\mathfrak{I}}^0(Y)$ be arbitrary. We shall construct an element $\Phi_{a^{**},b^{**}}\in \Delta_{\mathfrak I}(X\widehat \otimes_\varepsilon Y)$ such that $J(\Phi_{a^{**},b^{**}})=a^{**}\otimes b^{**}$. 

The strategy is as follows. We use the canonical operator $T:X\widehat{\otimes}_{\pi}Y\to X\widehat \otimes_\varepsilon Y$,
induced by the bilinear map $(x,y)\mapsto x\otimes y$. Given
$\varphi\in (X\widehat \otimes_\varepsilon Y)^*$, we regard $T^*(\varphi)$ as an operator from
$X$ to $Y^*$, evaluate it through $b^{**}$ to obtain a functional on
$X$, and then evaluate this functional through $a^{**}$.

Under the canonical isometric identification $(X\widehat{\otimes}_{\pi}Y)^*
\cong \mathcal L(X,Y^*)$, for each $\varphi\in (X\widehat \otimes_\varepsilon Y)^*$, we have the  operator $T^*(\varphi):X\to Y^*$ given by 
\[(T^*(\varphi))(x)(y)=\varphi(x\otimes y).\]

For each such $\varphi$, the map $V_{b^{**}}^\varphi:X\to \mathbb{R}$ given by $V_{b^{**}}^{\varphi}=b^{**}\circ T^*(\varphi)$ is an element of $X^*$. We set
\[
\Phi_{a^{**},b^{**}}(\varphi)=a^{**}(V_{b^{**}}^\varphi).
\]
Clearly, $|\Phi_{a^{**},b^{**}}(\varphi)|\leq\|a^{**}\|\,\|b^{**}\|\,\|\varphi\|$, so that $\Phi_{a^{**},b^{**}}\in (X\widehat \otimes_\varepsilon Y)^{**}$. We claim that $\Phi_{a^{**},b^{**}}\in \Delta_{\mathfrak{I}}^0(X\widehat \otimes_\varepsilon Y)$. Let $A\in \mathfrak{I}((X\widehat \otimes_\varepsilon Y)^*)$ be arbitrary. By functoriality it follows that $\{T^*(\varphi):\varphi \in A\}$ is an element of  $\mathfrak{I}(\mathcal{L}(X,Y^*))$. Then, by condition (3) in the definition of testing ideal, we obtain 
\[\{V_{b^{**}}^{\varphi}:\varphi \in A\}=\{b^{**}\circ  T^*(\varphi):\varphi \in A\} \in \mathfrak I(X^*).\]

Hence, since $a^{**}\in \Delta_{\mathfrak{I}}^0(X)$, there exists  $a\in X$ such that 
\[
a^{**}(V_{b^{**}}^\varphi)=\Phi_{a^{**},b^{**}}(\varphi)= V_{b^{**}}^{\varphi}(a)
\]
for each $\varphi \in A$. For this fixed $a\in X$, consider the operator $I_a:Y\to X\widehat\otimes_\pi Y$, given by $I_a(y)=a\otimes y$.
Under the canonical identification $(X\widehat\otimes_\pi Y)^* \cong \mathcal L(X,Y^*)$, we have $I_a^*(S)=S(a)$
for every $S\in\mathcal L(X,Y^*)$. Therefore, by functoriality,
\[\{T^*(\varphi)(a):\varphi\in A\}
=
\{I_a^*(T^*(\varphi)):\varphi\in A\}
\in \mathfrak I(Y^*).
\]
Since $b^{**}\in \Delta_{\mathfrak I}^0(Y)$, there exists $b\in Y$ such that 
\[
\Phi_{a^{**},b^{**}}(\varphi)= V_{b^{**}}^{\varphi}(a)=b^{**}(T^*(\varphi)(a))=(T^*(\varphi))(a)(b)=\varphi(a\otimes b)
\]
for each $\varphi \in A$. Since $a\otimes b\in X\widehat \otimes_\varepsilon Y$ we obtain that  $\Phi_{a^{**},b^{**}}\in \Delta_{\mathfrak{I}}^0(X\widehat \otimes_\varepsilon Y)$.

To complete the proof we show that $J(\Phi_{a^{**},b^{**}})=a^{**}\otimes b^{**}$. Indeed, given $x^*\in X^*$ and $y^*\in Y^*$, let $\varphi=\widehat{G}(x^*\otimes y^*)\in (X\widehat \otimes_\varepsilon Y)^*$. Then the associated operator $T^*(\varphi):X\to Y^*$ is given by the formula $T^*(\varphi)(x)=x^*(x)y^*$. Hence, for every $x\in X$, 
\[
V_{b^{**}}^{\varphi}(x)=b^{**}(T^*(\varphi)(x))=x^*(x) b^{**}(y^*).
\]
Therefore,
\begin{align*}
J(\Phi_{a^{**},b^{**}})(x^*\otimes y^*)&=\Phi_{a^{**},b^{**}}(\widehat{G}(x^*\otimes y^*))\\
&=a^{**}\bigl(x^*\, b^{**}(y^*)\bigr)=a^{**}(x^*)b^{**}(y^*)=(a^{**}\otimes b^{**})(x^*\otimes y^*).
\end{align*}

Since elementary tensors in $X^*\otimes Y^*$ separate points of
$\Delta_{\mathfrak I}(X)\widehat\otimes_\varepsilon
\Delta_{\mathfrak I}(Y)$, it follows that $J(\Phi_{a^{**},b^{**}})=a^{**}\otimes b^{**}$. Finally, by linearity, the range of $J$ contains
$\Delta_{\mathfrak{I}}^0(X)\otimes \Delta_{\mathfrak{I}}^0(Y)$.
\end{proof}

We are now ready to state the main result of this section: a symmetric
tensor-product formula for the derivative associated with a testing
ideal $\mathfrak I$.

\begin{theorem}\label{Thm:ProductRule}  Let $\mathfrak I$ be a testing ideal and $\Delta_{\mathfrak I}$ the associated ideal-generated assignment.
Let $X$ and $Y$ be Banach spaces satisfying all the hypotheses of Theorem~\ref{Thm:TensorEmbeddingIntersection}. Suppose further that $X$ is complemented in $\Delta_{\mathfrak{I}}(X)$ and that $Y$ is complemented in $\Delta_{\mathfrak{I}}(Y)$. Then
\[
\mathcal{S}_{\mathfrak{I}}(X\widehat\otimes_\varepsilon Y)\sim (X\widehat\otimes_\varepsilon \mathcal{S}_{\mathfrak{I}}(Y))
\oplus
\bigl(\mathcal{S}_{\mathfrak{I}}(X)\widehat{\otimes}_{\varepsilon}Y\bigr)
\oplus
\bigl(\mathcal{S}_{\mathfrak{I}}(X)
\widehat{\otimes}_{\varepsilon}
\mathcal{S}_{\mathfrak{I}}(Y)\bigr).
\]
\end{theorem}
\begin{proof}
According to Theorem~\ref{Thm:TensorEmbeddingIntersection}, the canonical mapping $J:\Delta_{\mathfrak{I}}(X\widehat\otimes_\varepsilon Y) \to \Delta_{\mathfrak{I}}(X)\widehat{\otimes}_{\varepsilon}\Delta_{\mathfrak{I}}(Y)$
is an isometric isomorphism and satisfies $J(f)=f$ for every
$f\in X\widehat\otimes_\varepsilon Y$.

As in the proof of Theorem~\ref{Thm:TensorQuotientDerivative}, we obtain
\[
\mathcal{S}_{\mathfrak{I}}(X\widehat\otimes_\varepsilon Y)
\cong
\frac{\Delta_{\mathfrak{I}}(X)\widehat{\otimes}_{\varepsilon}
\Delta_{\mathfrak{I}}(Y)}
{X\widehat{\otimes}_{\varepsilon}Y}.
\]

Since $X$ is complemented in $\Delta_{\mathfrak{I}}(X)$, let
$P:\Delta_{\mathfrak{I}}(X)\to X$ be a bounded projection and put
$E=\ker P$. Then $\Delta_{\mathfrak{I}}(X)=X\oplus E$
and
\[
E\sim \Delta_{\mathfrak{I}}(X)/X
=\mathcal{S}_{\mathfrak{I}}(X).
\]

Similarly, since $Y$ is complemented in $\Delta_{\mathfrak{I}}(Y)$, let
$Q:\Delta_{\mathfrak{I}}(Y)\to Y$ be a bounded projection and put
$Z=\ker Q$. Then $\Delta_{\mathfrak{I}}(Y)=Y\oplus Z$,
and therefore
\[
Z\sim \Delta_{\mathfrak{I}}(Y)/Y
=\mathcal{S}_{\mathfrak{I}}(Y).
\]

Now, the operator $P\otimes Q:\Delta_{\mathfrak I}(X)\widehat{\otimes}_\varepsilon\Delta_{\mathfrak I}(Y)\to X\widehat{\otimes}_\varepsilon Y$
is a bounded projection, see \cite[Proposition 3.2]{Ryan}. Since $\Delta_{\mathfrak{I}}(X)=X\oplus E$ and $\Delta_{\mathfrak{I}}(Y)=Y\oplus Z$,
we have
\[\Delta_{\mathfrak I}(X)\widehat{\otimes}_\varepsilon\Delta_{\mathfrak I}(Y)= (X\widehat{\otimes}_{\varepsilon}Y)\oplus(X\widehat{\otimes}_{\varepsilon}Z)\oplus (E\widehat{\otimes}_{\varepsilon}Y) \oplus (E\widehat{\otimes}_{\varepsilon}Z).\]
Hence
\[
\ker(P\otimes Q)=(X\widehat{\otimes}_{\varepsilon}Z)\oplus
(E\widehat{\otimes}_{\varepsilon}Y)
\oplus
(E\widehat{\otimes}_{\varepsilon}Z).
\]

Consequently
\begin{align*}
\mathcal{S}_{\mathfrak{I}}(X\widehat\otimes_\varepsilon Y)
\sim\;&
\frac{\Delta_{\mathfrak{I}}(X)\widehat{\otimes}_{\varepsilon}
\Delta_{\mathfrak{I}}(Y)}
{X\widehat{\otimes}_{\varepsilon}Y} \sim
X\widehat{\otimes}_{\varepsilon}Z
\oplus
E\widehat{\otimes}_{\varepsilon}Y
\oplus
E\widehat{\otimes}_{\varepsilon}Z
\\
\sim\;&
\bigl(X\widehat\otimes_\varepsilon \mathcal{S}_{\mathfrak{I}}(Y)\bigr)\oplus
\bigl(\mathcal{S}_{\mathfrak{I}}(X)
\widehat{\otimes}_{\varepsilon}Y\bigr)
\oplus
\bigl(\mathcal{S}_{\mathfrak{I}}(X)
\widehat{\otimes}_{\varepsilon}
\mathcal{S}_{\mathfrak{I}}(Y)\bigr).
\end{align*}
\end{proof}


\section{Derivatives of $C(K,Y)$ spaces}
\label{Sec:DerivativesC(K,Y)}

In this section, we investigate the hypotheses of
Theorems~\ref{Thm:DeltaInjectiveTensor},
\ref{Thm:CKTensorQuotient},
\ref{Thm:TensorEmbeddingIntersection}, and
\ref{Thm:ProductRule} in the setting of spaces of the form
\[
C(K,Y)\cong C(K)\widehat{\otimes}_{\varepsilon}Y.
\]
Our purpose is to identify natural classes of compact spaces $K$ and
Banach spaces $Y$ for which these hypotheses are satisfied. Two classes
of compact spaces emerge prominently in this analysis: scattered compact
spaces and finite products of compact lines. The results established
here will provide the framework for the applications developed in the
next section.

The section is organized as follows. We first verify the sequential
compactness and weak$^*$-sequential density assumptions required by
Theorems~\ref{Thm:DeltaInjectiveTensor} and
\ref{Thm:TensorEmbeddingIntersection}. We then study ideal-generated
assignments on spaces $C(K)$ and obtain a canonical representation of
$\Delta_{\mathfrak I}(C(K))$ as a Banach lattice of bounded functions
on $K$. Finally, we introduce a sufficient condition for $C(K)$ to be
complemented in $\Delta_{\mathfrak I}(C(K))$ and use it to derive
representations of the associated derivative and formulas for finite
products.

Before turning to the more delicate sequential conditions, we note
that one of the hypotheses of
Theorem~\ref{Thm:TensorEmbeddingIntersection} is automatically
satisfied for spaces of the form $C(K)$. Indeed, for every compact
Hausdorff space $K$, the bidual $C(K)^{**}$ is isometrically lattice
isomorphic to a space $C(L)$ for some compact Hausdorff space $L$; see
\cite{Kakutani}. Consequently, $C(K)^{**}$ has the approximation
property.

\subsection{Sequential properties and canonical tensor mappings}

We first examine the weak$^*$-sequential compactness of dual unit
balls, as required in Theorem~\ref{Thm:TensorEmbeddingIntersection}.
For scattered compact spaces, this property follows from the theory of
Radon--Nikodým compact spaces. Indeed, if $K$ is scattered, then
$(B_{C(K)^*},w^*)$ is a Radon--Nikodým compact space; see
\cite[Section~C, c-15]{EGT}. Hence it is weak$^*$-sequentially compact
by \cite[Corollary~5.4]{Namioka}.

Establishing the corresponding result for finite products of compact
lines requires a more detailed argument based on Rosenthal compact
spaces. Recall that a compact space is called Rosenthal if it is
homeomorphic to a compact subspace of $B_1(P)$, the space of
real-valued functions of the first Baire class on some Polish space
$P$, endowed with the topology of pointwise convergence; see
\cite[Section~C, c-17]{EGT}.

Certain lexicographically ordered compacta are known to be Rosenthal;
see, for instance,
\cite[Proposition~5]{AvilesTodorcevicLexicographic}. We shall need the
following elementary variant for spaces obtained by splitting the
points of a closed subset.

\begin{lemma}\label{Lem:ClosedSplitRosenthal}
Let $F\subset[0,1]$ be a compact set and let $A\subset F$ be closed. Set $\mathcal{D}(F,A)=\bigl(F\times\{0\}\bigr)\cup\bigl(A\times\{1\}\bigr)$, endowed with the lexicographic order and its order topology. Then $\mathcal{D}(F,A)$ is a Rosenthal compact space.
\end{lemma}
\begin{proof}
Since $F$ is compact and $A$ is closed in $F$, the set $A$ is compact
and metrizable. Let $P=A\sqcup\{\infty\}$ be the topological disjoint
union, so that $\infty$ is an isolated point of $P$. Then $P$ is a compact metrizable space
and, in particular, a Polish space. For every $(t,i)\in (F\times\{0\})\cup(A\times\{1\})$, define
$f_{(t,i)}:P\to\mathbb R$ by
\[
f_{(t,i)}(a)=
\begin{cases}
t, & a=\infty,\\
1, & a\in A \text{ and } (t,i)\leq_{\mathrm{lex}}(a,0),\\
0, & a\in A \text{ and } (a,0)<_{\mathrm{lex}}(t,i).
\end{cases}
\]
Equivalently, $f_{(t,0)}=\mathbf 1_{A\cap[t,1]}+t\mathbf 1_{\{\infty\}}$ and $f_{(t,1)}=\mathbf 1_{A\cap(t,1]}+t\mathbf 1_{\{\infty\}}$. Since $A\cap[t,1]$ and $A\cap(t,1]$ are, respectively, closed and open in the metrizable space $A$, and $\infty$ is isolated, it follows that $f_{(t,\varepsilon)}\in B_1(P)$ for every $(t,\varepsilon)\in \mathcal{D}(F,A)$.

Define $\Theta:\mathcal{D}(F,A)\to B_1(P)$ by the formula 
$\Theta(t,\varepsilon)=f_{(t,\varepsilon)}$.
We regard $B_1(P)$ as a subspace of $\mathbb R^P$ endowed with the
topology of pointwise convergence.

We first show that $\Theta$ is injective. Indeed, if $f_{(s,\varepsilon)}=f_{(t,\eta)}$, then $s=f_{(s,\varepsilon)}(\infty)=f_{(t,\eta)}(\infty)=t$. If $t\in A$, since $f_{(t,0)}(t)=1$ and $f_{(t,1)}(t)=0$, we deduce that $\varepsilon=\eta$. Therefore $\Theta$ is injective.

To prove that $\Theta$ is continuous, since $B_1(P)$ carries the
topology of pointwise convergence, it is enough to show that, for each $z\in P$, the mapping $(t,\varepsilon)\mapsto f_{(t,\varepsilon)}(z)$
is continuous on $\mathcal{D}(F,A)$.

For $z=\infty$, this mapping is the natural projection $(t,\varepsilon)\mapsto t$, which is continuous. Now fix $a\in A$. From definitions above we have
\[
\left\{(t,\varepsilon)\in \mathcal{D}(F,A):f_{(t,\varepsilon)}(a)=1\right\}=\bigl\{x\in \mathcal{D}(F,A):x\leq(a,0)\bigr\}.
\]
Since $(a,1)$ is the successor of $(a,0)$, the interval $\bigl\{x\in \mathcal{D}(F,A):x\leq(a,0)\bigr\}$
is a clopen subset of $\mathcal{D}(F,A)$. Therefore, the mapping $(t,\varepsilon)\mapsto f_{(t,\varepsilon)}(a)$
is continuous. This proves that $\Theta$ is continuous.

Finally, since $\mathcal{D}(F,A)$ is compact and $B_1(P)$ is Hausdorff, $\Theta$ is a homeomorphism of $\mathcal{D}(F,A)$ onto its image.
It follows that $\Theta(\mathcal{D}(F,A))$ is a compact subset of $B_1(P)$. Therefore, $\mathcal{D}(F,A)$ is a Rosenthal compact space.
\end{proof}

\begin{lemma}\label{Lem:RosenthalMeasureBall}
Let $K$ be a Rosenthal compact space. Then $\bigl(B_{\mathcal M(K)},w^*\bigr)$ is Fréchet--Urysohn.
\end{lemma}
\begin{proof}
Set $L=(K\times\{-1,1\})\sqcup\{\infty\}$, where $\sqcup$ denotes the
topological disjoint union and $\infty$ is an isolated point of $L$.
Then $L$ is a Rosenthal compact space. By
\cite[Proposition~7]{GodefroyRosenthal},
$\bigl(\mathcal P(L),w^*\bigr)$ is a Rosenthal compact space. By the
Bourgain--Fremlin--Talagrand theorem
\cite[Theorem~3F]{BourgainFremlinTalagrand},
$\bigl(\mathcal P(L),w^*\bigr)$ is angelic
\cite[Section~K, p.~452]{EGT} and hence, being compact, is
Fréchet--Urysohn.
 
Define $R:\mathcal P(L)\to B_{\mathcal M(K)}$ by the formula
\[
R(\nu)(f)=\int_{K\times\{-1,1\}}\varepsilon f(t)\,d\nu(t,\varepsilon).
\]
The mapping $R$ is weak$^*$-continuous. Moreover, it is surjective.
Indeed, given $\mu\in B_{\mathcal M(K)}$, write
$\mu=\mu^+-\mu^-$ for its Jordan decomposition and set
\[
\nu=\mu^+\otimes\delta_{1}+\mu^-\otimes\delta_{-1}+(1-\|\mu\|)\delta_{\{\infty\}}.
\]
Then $\nu\in\mathcal P(L)$ and $R(\nu)=\mu$. Therefore
$B_{\mathcal M(K)}$ is a continuous image of the compact
Fréchet--Urysohn space $\mathcal P(L)$ and is consequently
Fréchet--Urysohn.
\end{proof}

\begin{proposition}\label{Prop:WeakStarSequentialCompactness}
Let $K$ be a finite product of compact lines. Then
$B_{C(K)^*}$ is weak$^*$-sequentially compact.
\end{proposition}
\begin{proof}
Write $K=\prod_{i=1}^m K_i$, where each $K_i$ is a compact line, and
identify $C(K)^*$ with $\mathcal M(K)$. Let $(\mu_n)_n$ be a sequence in
$B_{\mathcal M(K)}$. By \cite[p.~86]{Mercourakis}
(see also \cite[Lemma~3.4]{Korpalski}), each measure $\mu_n$ has
separable support.

For each $1\leq i\leq m$, let $\pi_i:K\to K_i$ denote the canonical
projection and set $K_i^0=\overline{\bigcup_{n\in\mathbb N}\pi_i(\operatorname{supp}\mu_n)}$.
Then $K_i^0$ is a separable compact line. Put $K_0=\prod_{i=1}^m K_i^0$. By Ostaszewski's representation theorem \cite{Ostaszewski}, for each
$1\leq i\leq m$, there exist a compact set $F_i\subset[0,1]$ and a
subset $A_i\subset F_i$ such that $K_i^0$ is order isomorphic, and hence
homeomorphic, to $\mathcal D(F_i,A_i)= \bigl(F_i\times\{0\}\bigr)\cup\bigl(A_i\times\{1\}\bigr)$,
endowed with the lexicographic order and its order topology.

Let $\overline{A_i}$ denote the closure of $A_i$ in $F_i$ and set $\mathcal D(F_i,\overline{A_i})=\bigl(F_i\times\{0\}\bigr)\cup\bigl(\overline{A_i}\times\{1\}\bigr)$, also endowed with the lexicographic order topology. Define $q_i:\mathcal D(F_i,\overline{A_i})\to\mathcal D(F_i,A_i)$
by the formula
\[
q_i(t,\varepsilon)=
\begin{cases}
(t,\varepsilon),
    & (t,\varepsilon)\in\mathcal D(F_i,A_i),\\
(t,0),
    & (t,\varepsilon)\in
      (\overline{A_i}\setminus A_i)\times\{1\}.
\end{cases}
\]
It is readily seen that $q_i$ is an order-preserving continuous surjection since it collapses each adjacent pair $(t,0)<(t,1)$, with
$t\in\overline{A_i}\setminus A_i$, to the point $(t,0)$ and leaves all other points fixed.

By Lemma~\ref{Lem:ClosedSplitRosenthal}, each $\mathcal D(F_i,\overline{A_i})$ is a Rosenthal compact space. Since
finite products of Rosenthal compact spaces are Rosenthal, $\widetilde K_0= \prod_{i=1}^m\mathcal D(F_i,\overline{A_i})$
is a Rosenthal compact space. After identifying each $K_i^0$ with
$\mathcal D(F_i,A_i)$, the product mapping $q=\prod_{i=1}^m q_i:\widetilde K_0\to K_0$
is a continuous surjection.

The composition operator $Q:C(K_0)\to C(\widetilde K_0)$, $Q(f)=f\circ q$, is a linear isometry. Therefore, by the Hahn--Banach theorem, its
adjoint $Q^*:\mathcal M(\widetilde K_0)\to\mathcal M(K_0)$ maps the closed unit ball onto the closed unit ball. Moreover,
$Q^*:\bigl(B_{\mathcal M(\widetilde K_0)},w^*\bigr)\to\bigl(B_{\mathcal M(K_0)},w^*\bigr)$ is continuous.

Since $\widetilde K_0$ is a Rosenthal compact space, it follows from
Lemma~\ref{Lem:RosenthalMeasureBall} that
$\bigl(B_{\mathcal M(\widetilde K_0)},w^*\bigr)$ is Fréchet--Urysohn. Since $Q^*$ maps this ball continuously onto
$B_{\mathcal M(K_0)}$, it follows that $\bigl(B_{\mathcal M(K_0)},w^*\bigr)$ is Fréchet--Urysohn and, in
particular, sequentially compact.

Since each $\mu_n$ is supported on $K_0$, we may regard $\mu_n$ as an
element of $B_{\mathcal M(K_0)}$. Hence, after passing to a subsequence,
there exists $\nu\in B_{\mathcal M(K_0)}$ such that
$\mu_{n_j}\xrightarrow{w^*}\nu$ in $\mathcal M(K_0)$.

Let $\iota:K_0\to K$ be the canonical inclusion and set $\widetilde\nu=\iota_*(\nu)\in B_{\mathcal M(K)}$. Since every $\mu_n$
is supported on $K_0$, for every $f\in C(K)$ we have
\[
\int_K f\,d\mu_{n_j}
=
\int_{K_0}f|_{K_0}\,d\mu_{n_j}
\to
\int_{K_0}f|_{K_0}\,d\nu
=
\int_K f\,d\widetilde\nu.
\]
Thus $\mu_{n_j}\xrightarrow{w^*}\widetilde\nu$ in $\mathcal M(K)$. Therefore, $B_{C(K)^*}$ is weak$^*$-sequentially compact.
\end{proof}

\begin{corollary}\label{Cor:ClassesGP}
Let $K$ be either a scattered compact Hausdorff space or a finite
product of compact lines. Then $C(K)$ has the Gelfand--Phillips
property.
\end{corollary}
\begin{proof}
In both cases, $B_{C(K)^*}$ is weak$^*$-sequentially compact. Therefore,
$C(K)$ has the Gelfand--Phillips property; see
\cite[Chapter~XIII, Exercise~4, p.~238]{Diestel}
(see also \cite{Gelfand1938}).
\end{proof}

We now turn to the weak$^*$-sequential density assumptions involving
the canonical operator $\widehat G:X^*\widehat{\otimes}_\pi Y^*\to (X\widehat{\otimes}_\varepsilon Y)^*$.
For scattered compact spaces, this density condition holds in a
particularly strong form: the canonical operator is an isometric
isomorphism.

\begin{proposition}\label{Prop:CanonicalScattered}
Let $K$ be a scattered compact Hausdorff space and $Y$ be a Banach space. Then the canonical mapping $\widehat G:C(K)^*\widehat{\otimes}_{\pi}Y^*\to (C(K)\widehat{\otimes}_{\varepsilon}Y)^*$ is an isometric isomorphism.
\end{proposition}
\begin{proof}
Since $C(K,Y)^*$ is isometrically isomorphic to $\ell_1(K,Y^*)$ for every Banach space $Y$ (\cite[p. 192]{Singer} and also \cite[Example 2.6]{Ryan}), we have the following chain of canonical isometries:
\begin{align*}
    C(K)^* \widehat{\otimes}_\pi Y^* \cong  \ell_1(K) \widehat{\otimes}_\pi Y^* \cong  \ell_1(K, Y^*) \cong  C(K,Y)^* \cong  (C(K) \widehat{\otimes}_\varepsilon Y)^*.
\end{align*}

We claim that the above chain of isometries agrees with the canonical mapping $\widehat G$ on $C(K)^*\widehat{\otimes}_{\pi}Y^*$. Indeed, let $\mu\in C(K)^*$ and $y^*\in Y^*$, and let $u=\sum_{j=1}^n f_j\otimes y_j\in C(K)\otimes Y$. On the one hand we have
\[
\widehat G(\mu\otimes y^*)(u)
=
\sum_{j=1}^n \mu(f_j)y^*(y_j).
\]

On the other hand, writing $\mu=\sum_{t\in K} a_t\delta_t$ (since $C(K)^*=\ell_1(K)$), under the canonical identification $\ell_1(K)\widehat{\otimes}_\pi Y^*=\ell_1(K,Y^*)$, the tensor $\mu\otimes y^*$ corresponds to the element $(a_t y^*)_{t\in K}\in \ell_1(K,Y^*)$, which, in turn, defines a functional on $C(K,Y)$ by the formula
\[
f\mapsto \sum_{t\in K} a_t\, y^*(f(t)).
\]
Applying this functional to $u=\sum_{j=1}^n f_j\otimes y_j$, we obtain
\[
\sum_{t\in K} a_t\, y^*\!\big(\sum_{j=1}^n f_j(t)y_j\big)
=
\sum_{j=1}^n \big(\sum_{t\in K} a_t f_j(t)\big) y^*(y_j)
=
\sum_{j=1}^n \mu(f_j)y^*(y_j).
\]

Thus both mappings coincide on elementary tensors, and hence on all of $C(K)^*\widehat{\otimes}_{\pi}Y^*$ by continuity. Hence $\widehat{G}$ is an isometric isomorphism.
\end{proof}

For finite products of compact lines, the canonical operator need not
be surjective. Nevertheless, its image contains enough elementary
tensors to satisfy the weak$^*$-sequential density condition required
in Theorem~\ref{Thm:TensorEmbeddingIntersection}. In fact, we prove the
stronger sequential approximation property stated below.

\begin{proposition}\label{Prop:CompactLinesAux}
Let $K$ and $L$ be finite products of compact lines, and let
$\widehat G:C(K)^*\widehat\otimes_\pi C(L)^*\to
(C(K)\widehat\otimes_\varepsilon C(L))^*$ be the canonical operator.
Then $\widehat G\bigl((C(K)^*\otimes C(L)^*)\cap
B_{C(K)^*\widehat\otimes_\pi C(L)^*}\bigr)$ is weak$^*$-sequentially dense in $B_{(C(K)\widehat\otimes_\varepsilon C(L))^*}$.
\end{proposition}
\begin{proof}
We prove the stronger statement that every element of
$B_{(C(K)\widehat\otimes_\varepsilon C(L))^*}$ is the weak$^*$-limit
of a sequence in $\widehat G\bigl((C(K)^*\otimes C(L)^*)\cap
B_{C(K)^*\widehat\otimes_\pi C(L)^*}\bigr)$.

We use the canonical identification
$(C(K)\widehat\otimes_\varepsilon C(L))^*
\cong\mathcal M(K\times L)$. Let
$\mu\in B_{\mathcal M(K\times L)}$ be arbitrary. Repeating the
construction in the proof of
Proposition~\ref{Prop:WeakStarSequentialCompactness} for the measure
$\mu$, we obtain a compact subspace
$K_0\times L_0\subset K\times L$ such that $\mu$ is supported on
$K_0\times L_0$ and
$\bigl(B_{\mathcal M(K_0\times L_0)},w^*\bigr)$ is
Fréchet--Urysohn.

Since $\mu$ is supported on $K_0\times L_0$, we may regard its
restriction to $K_0\times L_0$ as an element of
$B_{\mathcal M(K_0\times L_0)}$. By the Krein--Milman theorem
\cite[Theorem~3.65]{Fabian}, the finitely supported measures are
weak$^*$-dense in this ball. Since
$\bigl(B_{\mathcal M(K_0\times L_0)},w^*\bigr)$ is
Fréchet--Urysohn, there exists a sequence $(\mu_n)_n$ of finitely
supported measures in $B_{\mathcal M(K_0\times L_0)}$, of the form
\[
\mu_n=\sum_{j=1}^{d_n}a_{n,j}\delta_{(s_{n,j},t_{n,j})},
\]
where $s_{n,j}\in K_0$, $t_{n,j}\in L_0$, and $\sum_{j=1}^{d_n}|a_{n,j}|\leq1$, such that $\mu_n\xrightarrow{w^*}\mu$ in
$\mathcal M(K_0\times L_0)$. Viewing each $\mu_n$ canonically as a measure on $K\times L$, we also have $\mu_n\xrightarrow{w^*}\mu$ in 
$\mathcal M(K\times L)$.

Note that, for every $(s,t)\in K\times L$, we have $\delta_{(s,t)}=\widehat G(\delta_s\otimes\delta_t)$. For each $n\in\mathbb N$, define
\[
\eta_n=\sum_{j=1}^{d_n}a_{n,j}\delta_{s_{n,j}}\otimes\delta_{t_{n,j}}.
\]
Then $\|\eta_n\|_\pi\leq\sum_{j=1}^{d_n}|a_{n,j}|\leq1$, and hence $\eta_n\in(C(K)^*\otimes C(L)^*)\cap B_{C(K)^*\widehat\otimes_\pi C(L)^*}$.
Moreover, since $\widehat G(\eta_n)=\mu_n$, we obtain $\widehat G(\eta_n)\xrightarrow{w^*}\mu$. This completes the proof.
\end{proof}

We next establish a vector-valued version of the preceding weak$^*$-sequential density result, with the second factor replaced by
a $\operatorname{WCG}$ Banach space.

\begin{proposition}\label{Prop:CompactLinesWCG}
Let $K$ be a finite product of compact lines, let $X$ be a $\operatorname{WCG}$ Banach space, and let $\widehat G:
C(K)^*\widehat\otimes_\pi X^* \to (C(K)\widehat\otimes_\varepsilon X)^*$ be the canonical operator. Then $\widehat G\bigl((C(K)^*\otimes X^*)\cap
B_{C(K)^*\widehat\otimes_\pi X^*}\bigr)$ is weak$^*$-sequentially dense in $B_{(C(K)\widehat\otimes_\varepsilon X)^*}$.
\end{proposition}
\begin{proof}
We prove the stronger statement that every element of
$B_{(C(K)\widehat\otimes_\varepsilon X)^*}$ is the weak$^*$-limit
of a sequence in $\widehat G\bigl((C(K)^*\otimes X^*)\cap
B_{C(K)^*\widehat\otimes_\pi X^*}\bigr)$.

Using the canonical isometric identification $C(K)\widehat\otimes_\varepsilon X=C(K,X)$, let
$M=(B_{X^*},w^*)$ and consider the canonical isometric embedding $J:C(K,X)\to C(K\times M)$
defined by
\[
J(f)(t,x^*)=x^*(f(t))
\]
for every $f\in C(K,X)$, $t\in K$, and $x^*\in M$.

Let $\Phi\in B_{C(K,X)^*}$ be arbitrary. By the Hahn--Banach theorem, there exists an extension $\widetilde\Phi\in B_{C(K\times M)^*}$
such that $\widetilde\Phi\circ J=\Phi$. Identifying
$C(K\times M)^*$ with $\mathcal M(K\times M)$, let
$\mu\in B_{\mathcal M(K\times M)}$ be the measure corresponding to
$\widetilde\Phi$.

Let $\pi_K:K\times M\to K$ and $\pi_M:K\times M\to M$ denote the canonical projections, and let $(\pi_K)_*|\mu|$ and $(\pi_M)_*|\mu|$
denote the corresponding push-forward measures. It follows from \cite[p.~86]{Mercourakis} (see also \cite[Lemma~3.4]{Korpalski}) that there exists a separable set
$S_K\subset K$ such that
\[
(\pi_K)_*|\mu|(S_K)=|\mu|(K\times M).
\]
Write $K=\prod_{i=1}^m K_i$, where each $K_i$ is a compact line, and let $\pi_i:K\to K_i$ denote the canonical projection. For each $1\leq i\leq m$, set $K_i^0=\overline{\pi_i(S_K)}$ and $K_0=\prod_{i=1}^m K_i^0$. Then each $K_i^0$ is a separable compact line and $|\mu|(K_0\times M)=|\mu|(K\times M)$.

Since $X$ is $\operatorname{WCG}$, the Amir--Lindenstrauss theorem \cite[Theorem 14.9]{Fabian} implies that
$M=(B_{X^*},w^*)$ is an Eberlein compact space. In particular, every Radon measure on $M$ has separable support \cite[p.~86]{Mercourakis}. Applying this fact to the measure $(\pi_M)_*|\mu|$, there exists a separable compact subspace
$M_0\subset M$ such that $|\mu|(K\times M_0)=|\mu|(K\times M)$. Since $|\mu|$ is concentrated both on $K_0\times M$ and on
$K\times M_0$, it is concentrated on their intersection
$K_0\times M_0$. Consequently,
\[
|\mu|(K_0\times M_0)=|\mu|(K\times M).
\]

Proceeding as in the proof of Proposition~\ref{Prop:WeakStarSequentialCompactness}, we obtain a
separable Rosenthal compact space $\widetilde K_0$ and a continuous
surjection $q:\widetilde K_0\to K_0$. Moreover, the separable Eberlein
compact space $M_0$ is metrizable and hence Rosenthal. Therefore,
$\widetilde K_0\times M_0$ is a separable Rosenthal compact space, and
the map $\widehat q=q\times\operatorname{Id}_{M_0}:\widetilde K_0\times M_0\to K_0\times M_0$
is a continuous surjection.

The composition operator $Q:C(K_0\times M_0)\to C(\widetilde K_0\times M_0)$, $Q(f)=f\circ\widehat q$,
is a linear isometric embedding. Consequently, its adjoint maps $B_{\mathcal M(\widetilde K_0\times M_0)}$ onto
$B_{\mathcal M(K_0\times M_0)}$ and is weak$^*$-to-weak$^*$ continuous. By Lemma~\ref{Lem:RosenthalMeasureBall},
$(B_{\mathcal M(\widetilde K_0\times M_0)},w^*)$ is Fréchet--Urysohn. Since continuous images of compact
Fréchet--Urysohn spaces are Fréchet--Urysohn, it follows that $(B_{\mathcal M(K_0\times M_0)},w^*)$ is Fréchet--Urysohn.

Let $\nu$ be the restriction of $\mu$ to $K_0\times M_0$. Then $\nu\in B_{\mathcal M(K_0\times M_0)}$. By the Krein--Milman theorem \cite[Theorem 3.65]{Fabian}, the finitely supported measures in
$B_{\mathcal M(K_0\times M_0)}$ are weak$^*$-dense in this ball.
Therefore, since $(B_{\mathcal M(K_0\times M_0)},w^*)$ is Fréchet--Urysohn, there exists a sequence $(\nu_n)_n$ of the form
\[
\nu_n=
\sum_{j=1}^{d_n}
a_{n,j}\delta_{(t_{n,j},x_{n,j}^*)},
\]
where $t_{n,j}\in K_0$, $x_{n,j}^*\in M_0$, $\sum_{j=1}^{d_n}|a_{n,j}|\leq1$, and $\nu_n\xrightarrow{w^*}\nu$ in $\mathcal M(K_0\times M_0)$. 

For each $n\in\mathbb N$, define
\[
\eta_n=
\sum_{j=1}^{d_n}
a_{n,j}\delta_{t_{n,j}}\otimes x_{n,j}^*
\in C(K)^*\otimes X^*.
\]
Since $\|x_{n,j}^*\|\leq1$, we have $\|\eta_n\|_\pi\leq\sum_{j=1}^{d_n} |a_{n,j}|\|\delta_{t_{n,j}}\|\|x_{n,j}^*\|\leq1$. Thus $\eta_n\in
(C(K)^*\otimes X^*)\cap B_{C(K)^*\widehat\otimes_\pi X^*}$.

For every $f\in C(K,X)$, we have
\begin{align*}
\widehat G(\eta_n)(f)
&=
\sum_{j=1}^{d_n}
a_{n,j}x_{n,j}^*(f(t_{n,j}))\\
&=
\int_{K_0\times M_0}J(f)\,d\nu_n
\to
\int_{K_0\times M_0}J(f)\,d\nu\\
&=
\int_{K\times M}J(f)\,d\mu
=
\widetilde\Phi(J(f))
=
\Phi(f).
\end{align*}
Therefore, $\widehat G(\eta_n)\xrightarrow{w^*}\Phi$ in $(C(K)\widehat\otimes_\varepsilon X)^*$. This completes the proof.
\end{proof}

\subsection{Ideal-generated assignments on $C(K)$ spaces}
We next study the structure of
$\Delta_{\mathfrak I}(C(K))$. Besides providing a concrete function
representation of this space, the results below will verify the
approximation property required in the tensor-product theorems.

Although $\Delta_{\mathfrak I}(C(K))$ is a closed subspace of
$C(K)^{**}$ and $C(K)^{**}$ has the approximation property, this alone
does not imply that $\Delta_{\mathfrak I}(C(K))$ has the approximation
property. We therefore introduce a compatibility condition that allows
the elements of $\Delta_{\mathfrak I}(C(K))$ to be represented by
bounded functions on $K$.

\begin{definition}
Let $\mathfrak I$ be a testing ideal and let $K$ be a compact Hausdorff space. Define $\mathcal I_{\mathfrak I}(K)=\bigl\{
A\subset K:\{\delta_a\:a\in A\}\in \mathfrak I(C(K)^*)\bigr\}$. We say that the compact space $K$ is \emph{$\mathfrak I$-compatible} if, for every
$B\in \mathfrak I(C(K)^*)$, there exists $A\in \mathcal I_{\mathfrak I}(K)$ such that
\[
|\mu|(\overline A)=|\mu|(K)
\]
for every $\mu\in B$.
\end{definition}

\begin{remark}\label{Rem:ExamplesCompatible}Recalling the examples from Remark~\ref{Rem:ExamplesAssignments}, every scattered compact Hausdorff space and every finite product of compact lines is $\mathfrak I_\kappa$-compatible for every infinite cardinal $\kappa$. Moreover, every scattered compact Hausdorff space and every finite product of compact lines is $\mathfrak I_{<\kappa}$-compatible whenever $\kappa$ is an uncountable regular cardinal.

Indeed, let $K$ be a scattered compact Hausdorff space and let
$B\in\mathfrak I(C(K)^*)$. Since $C(K)^*\cong \ell_1(K)$, every measure
$\mu\in B$ has countable support. Therefore, if
\[
A=\bigcup_{\mu\in B}\operatorname{supp}(\mu),
\]
then $A\in\mathcal I_{\mathfrak I}(K)$ and
$|\mu|(\overline A)=|\mu|(K)$ for every $\mu\in B$.

Now suppose that $K$ is a finite product of compact lines. By \cite[p.~86]{Mercourakis} (see also
\cite[Lemma~3.4]{Korpalski}), every measure $\mu\in C(K)^*$ admits a
separable set $S_\mu\subset K$ of full $|\mu|$-measure. If
$\mathfrak I=\mathfrak I_\kappa$, choose a countable dense subset
$D_\mu\subset S_\mu$ for each $\mu\in B$, and define
\[
A=\bigcup_{\mu\in B}D_\mu.
\]
Since $|B|\le\kappa$ and each $D_\mu$ is countable, it follows that
$|A|\le\kappa$, and therefore $A\in\mathcal I_\kappa(K)$.
Moreover, $S_\mu\subset \overline{D_\mu}\subset \overline A$,
so that $|\mu|(\overline A)=|\mu|(K)$ for every $\mu\in B$.
The argument for $\mathfrak I_{<\kappa}$ is analogous, using the regularity of
$\kappa$.
\end{remark}

The compatibility condition allows elements of $\Delta_{\mathfrak I}(C(K))$ to be represented by their 
values on point masses. This motivates the following function space.

\begin{definition}
Let $\mathfrak I$ be a testing ideal and let $K$ be a compact Hausdorff
space. Define
\[
\mathscr C_{\mathfrak I}(K)
=
\left\{
f\in\ell_\infty(K):
f|_{\overline A}\in C(\overline A)
\text{ for every }A\in\mathcal I_{\mathfrak I}(K)
\right\},
\]
endowed with the supremum norm. For each $\Phi\in C(K)^{**}$, define
$g_\Phi:K\to\mathbb R$ by the formula
\[
g_\Phi(t)=\Phi(\delta_t).
\]
\end{definition}

The following proposition was inspired by
\cite[Theorem~3.3]{Korpalski}.

\begin{proposition}\label{Prop:Fecho}
Let $\mathfrak I$ be a testing ideal and let $K$ be an
$\mathfrak I$-compatible compact Hausdorff space. Then the mapping $\Gamma:\Delta_{\mathfrak I}(C(K))\to
\mathscr C_{\mathfrak I}(K)$, defined by $\Gamma(\Phi)=g_\Phi$, is a canonical linear isometric isomorphism. Its inverse is given by $\Gamma^{-1}(f)=\Phi_f$, where
\[
\Phi_f(\mu)=\int_K f\,d\mu.
\]
\end{proposition}
\begin{proof}
We first show that $\Gamma$ is well defined. Let
$\Phi\in\Delta_{\mathfrak I}^0(C(K))$ and
$A\in\mathcal I_{\mathfrak I}(K)$. Since $\{\delta_a:a\in A\}\in\mathfrak I(C(K)^*)$,
there exists $f\in C(K)$ such that $g_\Phi(a)=f(a)$ for every $a \in A$. 

Now let $x \in \overline{A}$ be arbitrary. Since $\mathfrak I(C(K)^*)$ is an ideal, the set $\{\delta_a:a\in A\}\cup \{\delta_x\}$ also belongs to $\mathfrak I(C(K)^*)$. Hence there exists a function $h\in C(K)$ such that $g_\Phi(x)=h(x)$ and $g_\Phi(a)=h(a)$ for every $a \in A$.  Since $f$ and $h$ are continuous on $\overline{A}$ and coincide on $A$, we have $f|_{\overline{A}}=h|_{\overline{A}}$. Then $g_\Phi(x)=h(x)=f(x)$, and we deduce that $g_\Phi|_{\overline{A}}=f|_{\overline{A}}$. Thus $g_\Phi\in\mathscr C_{\mathfrak I}(K)$.

Now let $\Phi\in\Delta_{\mathfrak I}(C(K))$. Choose a sequence
$(\Phi_n)_n$ in $\Delta_{\mathfrak I}^0(C(K))$ such that
$\|\Phi_n-\Phi\|\to0$. For every $n\in\mathbb N$,
\[
\|g_{\Phi_n}-g_\Phi\|_\infty
=
\sup_{t\in K}|(\Phi_n-\Phi)(\delta_t)|
\leq
\|\Phi_n-\Phi\|.
\]
Therefore $g_{\Phi_n}\to g_\Phi$ uniformly on $K$. Since
$g_{\Phi_n}|_{\overline A}$ is continuous for every
$A\in\mathcal I_{\mathfrak I}(K)$, it follows that
$g_\Phi|_{\overline A}$ is continuous. Hence
$g_\Phi\in\mathscr C_{\mathfrak I}(K)$, and $\Gamma$ is well defined.
Its linearity is immediate, and the preceding inequality also gives $\|\Gamma(\Phi)\|_\infty=\|g_\Phi\|_\infty\leq\|\Phi\|$.

To construct the inverse mapping $\Gamma^{-1}$ we first note that each $f\in\mathscr C_{\mathfrak I}(K)$ is measurable with respect to every Radon measure on $K$. Indeed, let $\mu\in\mathcal M(K)$. Since $\{\mu\}$ belongs to $\mathfrak I(C(K)^*)$, the $\mathfrak I$-compatibility of $K$ yields
$A\in\mathcal I_{\mathfrak I}(K)$ such that
\[
|\mu|(\overline A)=|\mu|(K).
\]
Since $f|_{\overline A}$ is continuous, the Tietze extension theorem
provides a function $h\in C(K)$ satisfying $h|_{\overline A}=f|_{\overline A}$.
Thus $f=h$ $|\mu|$-almost everywhere, and hence $f$ is
$\mu$-measurable. We may therefore define $\Phi_f:\mathcal M(K)\to \mathbb{R}$ by the formula
\[
\Phi_f(\mu)=\int_K f\,d\mu.
\]
Moreover, $|\Phi_f(\mu)|\leq\|f\|_\infty\|\mu\|$, so $\Phi_f\in C(K)^{**}$ and
$\|\Phi_f\|\leq\|f\|_\infty$.

We claim that $\Phi_f\in\Delta_{\mathfrak I}^0(C(K))$. Let
$B\in\mathfrak I(C(K)^*)$. By the $\mathfrak I$-compatibility of $K$,
there exists $A\in\mathcal I_{\mathfrak I}(K)$ such that 
\[|\mu|(\overline A)=|\mu|(K)\]
for each $\mu\in B$. By Tietze extension theorem, choose $h\in C(K)$ extending $f|_{\overline A}$. Then, for every
$\mu\in B$,
\[
\Phi_f(\mu)
=
\int_K f\,d\mu
=
\int_{\overline A}f\,d\mu
=
\int_{\overline A}h\,d\mu
=
\mu(h).
\]
Therefore $\Phi_f|_B=h|_B$, and hence $\Phi_f\in\Delta_{\mathfrak I}^0(C(K)) \subset\Delta_{\mathfrak I}(C(K))$.

For every $t\in K$, we have
\[
g_{\Phi_f}(t)
=
\Phi_f(\delta_t)
=
\int_K f\,d\delta_t
=
f(t).
\]
Thus $g_{\Phi_f}=f$. In particular, $\|\Phi_f\|\geq\sup_{t\in K}|\Phi_f(\delta_t)|=\sup_{t\in K}|f(t)|=\|f\|_\infty$.
Together with the reverse inequality obtained above, this gives $\|\Phi_f\|=\|f\|_\infty$.

It remains to prove that $\Phi_{g_\Phi}=\Phi$ for every $\Phi\in\Delta_{\mathfrak I}(C(K))$. We first consider $\Phi\in\Delta_{\mathfrak I}^0(C(K))$. Fix
$\mu\in\mathcal M(K)$. Applying the compatibility assumption to
$\{\mu\}$, choose $A\in\mathcal I_{\mathfrak I}(K)$ such that
\[
|\mu|(\overline A)=|\mu|(K).
\]
Since $\{\mu\}\cup\{\delta_a:a\in A\}\in\mathfrak I(C(K)^*)$, there exists $h\in C(K)$ such that $\Phi(\mu)=\mu(h)$
and $g_\Phi(a)=\Phi(\delta_a)=h(a)$ for each $a\in A$. We have already proved that $g_\Phi|_{\overline A}$ is continuous.
Since $g_\Phi$ and $h$ coincide on $A$, they coincide on
$\overline A$. Therefore,
\[
\Phi(\mu)=\mu(h)=\int_{\overline A}h\,d\mu=
\int_{\overline A}g_\Phi\,d\mu=\int_Kg_\Phi\,d\mu=\Phi_{g_\Phi}(\mu).
\]
Since $\mu$ was arbitrary, $\Phi=\Phi_{g_\Phi}$.

Finally, let $\Phi\in\Delta_{\mathfrak I}(C(K))$ and choose
$(\Phi_n)_n$ in $\Delta_{\mathfrak I}^0(C(K))$ such that
$\|\Phi_n-\Phi\|\to0$. Since $g_{\Phi_n}\to g_\Phi$ uniformly, for
every $\mu\in\mathcal M(K)$ we have
\[
\Phi(\mu)=\lim_{n\to\infty}\Phi_n(\mu)=\lim_{n\to\infty}\int_Kg_{\Phi_n}\,d\mu=\int_Kg_\Phi\,d\mu=\Phi_{g_\Phi}(\mu).
\]
Thus $\Phi=\Phi_{g_\Phi}$. Consequently, the mappings
$\Phi\mapsto g_\Phi$ and $f\mapsto\Phi_f$ are mutually inverse linear
isometries.
\end{proof}

\begin{remark}\label{Rem:Fecho}
In the applications of Proposition~\ref{Prop:Fecho}, it will often be more convenient to work with the associated functions $g_\Phi(t)=\Phi(\delta_t)$,
rather than with the corresponding elements $\Phi\in C(K)^{**}$.

In the particular case when $K$ is a scattered compact Hausdorff space,
the canonical identification
$C(K)^{**}\cong \ell_\infty(K)$
shows that Proposition~\ref{Prop:Fecho} reduces to a concrete description in terms of bounded functions. Combining Proposition~\ref{Prop:Fecho} with Remark~\ref{Rem:ExamplesCompatible}, we obtain the following formulas for the testing ideals $\mathfrak I_\kappa$ and $\mathfrak I_{<\kappa}$.  Under this identification, for every infinite cardinal $\kappa$,
\[
\Delta_{\kappa}(C(K))
=
\Bigl\{
f\in \ell_\infty(K):
f|_{\overline A}\in C(\overline A)
\text{ whenever }
A\subset K
\text{ and }
|A|\le \kappa
\Bigr\}.
\]

Likewise, for every uncountable regular cardinal $\kappa$,
\[
\Delta_{<\kappa}(C(K))
=
\Bigl\{
f\in \ell_\infty(K):
f|_{\overline A}\in C(\overline A)
\text{ whenever }
A\subset K
\text{ and }
|A|<\kappa
\Bigr\}.
\]
\end{remark}

The preceding function representation also provides a natural Banach lattice structure on $\Delta_{\mathfrak I}(C(K))$.

\begin{theorem}\label{Thm:ScatteredRepresentation}
Let $\mathfrak I$ be a testing ideal and let $\Delta_{\mathfrak I}$ be the associated ideal-generated assignment. For every $\mathfrak I$-compatible compact Hausdorff space $K$, there exists a compact Hausdorff space $L$ such that $\Delta_{\mathfrak I}(C(K))
\cong C(L)$ isometrically as Banach lattices. Consequently, there exists a continuous surjection $\pi:L\to K$.
\end{theorem}
\begin{proof}
By Proposition~\ref{Prop:Fecho}, the mapping $\Gamma:\Delta_{\mathfrak I}(C(K))\to \mathscr C_{\mathfrak I}(K)$, $\Gamma(\Phi)=g_\Phi$,
is a canonical linear isometric isomorphism, where
\[
\mathscr C_{\mathfrak I}(K)=\left\{f\in\ell_\infty(K):f|_{\overline A}\in C(\overline A)
\text{ for every }A\in\mathcal I_{\mathfrak I}(K)
\right\}.
\]
The space $\mathscr C_{\mathfrak I}(K)$ is a norm-closed sublattice of
$\ell_\infty(K)$ containing the constant functions. We transport its
pointwise lattice operations to $\Delta_{\mathfrak I}(C(K))$ through
$\Gamma$.  With this lattice structure, $\Gamma$ is a unital isometric lattice
isomorphism by construction.

Since $\mathscr C_{\mathfrak I}(K)$ is an AM-space with unit,
Kakutani's representation theorem \cite[Theorem~1]{Kakutani} yields a
compact Hausdorff space $L$ and a unital isometric lattice isomorphism $U:\mathscr C_{\mathfrak I}(K)\to C(L)$. Consequently, $J=U\circ\Gamma:
\Delta_{\mathfrak I}(C(K))\to C(L)$ is a unital isometric lattice isomorphism.

Notice that $\Gamma(f)=f$ for every $f\in C(K)$. Hence the restriction $J|_{C(K)}:C(K)\to C(L)$ is a unital isometric lattice embedding.  For each
$y\in L$, the functional $\varphi_y\in C(K)^*$ given by $\varphi_y(f)=J(f)(y)$
is positive, satisfies $\varphi_y(\mathbf 1)=1$, and $\varphi_y(f^2)=\varphi_y(f)^2$
for every $f\in C(K)$. Let $\mu\in \mathcal M(K)$ be the Radon measure
associated with $\varphi_y$. Then
\[
\int_K \big(f-\int_K f\,d\mu\big)^2\,d\mu=0
\]
for every $f\in C(K)$. It follows that $\mu$ is supported at a single point.
Hence there exists a unique $x_y\in K$ such that $\mu=\delta_{x_y}$, and
therefore $J(f)(y)=f(x_y)$ for every $f\in C(K)$. Hence, we may define a function $\pi:L\to K$ by $\pi(y)=x_y$. This function satisfies 
\[
J(f)(y)=f(\pi(y))
\]
for each $y \in L$ and $f\in C(K)$. Consequently, $\pi$ is continuous. To prove that $\pi$ is surjective assume, towards a contradiction, that $\pi(L)\neq K$. Since $\pi(L)$ is compact, there exists $f\in C(K)$, $f\neq 0$, such that $f|_{\pi(L)}=0$. Hence $J(f)=f\circ\pi=0$, contradicting the injectivity of $J$.
\end{proof}

The following is an immediate consequence of Theorem~\ref{Thm:ScatteredRepresentation}.

\begin{corollary}\label{Cor:AP}
Let $\mathfrak I$ be a testing ideal and let $\Delta_{\mathfrak I}$ be the associated ideal-generated assignment. If $K$ is an $\mathfrak I$-compatible compact Hausdorff space, then $\Delta_{\mathfrak I}(C(K))$ has the approximation property.
\end{corollary}
\begin{proof}
By Theorem~\ref{Thm:ScatteredRepresentation},
$\Delta_{\mathfrak I}(C(K))$ is isometrically isomorphic to a space
$C(L)$, which has the approximation property.
\end{proof}

The representation provided by Theorem~\ref{Thm:ScatteredRepresentation} motivates the following definition.

\begin{definition}
Let $\mathfrak I$ be a testing ideal and let
$\Delta_{\mathfrak I}$ be the associated ideal-generated assignment.
Let $K$ be an $\mathfrak I$-compatible compact Hausdorff space.
A compact Hausdorff space $L$ is called a
\emph{$\Delta_{\mathfrak I}$-spectrum} of $K$ if $\Delta_{\mathfrak I}(C(K))\cong C(L)$ isometrically as Banach lattices.
\end{definition}

\subsection{Complementation and product formulas}

Although Theorem~\ref{Thm:CKTensorQuotient} removes the
complementation assumption in the asymmetric situation where the
first tensor factor is a $C(K)$ space satisfying the hypotheses of
Theorem~\ref{Thm:DeltaInjectiveTensor}, our proof of the full symmetric
product formula relies on complementation. We therefore study when
$C(K)$ is complemented in $\Delta_{\mathfrak I}(C(K))$. 

The construction below is inspired by ideas appearing in
\cite{CandidoSquares} for ordinal intervals and in
\cite{Korpalski} for compact lines. We isolate an extension condition
that is sufficient for the complementability of $C(K)$ in
$\Delta_{\mathfrak I}(C(K))$ and formulate it in the general framework
of testing ideals.

For a dense subset $D\subset K$, a function $f:K\to\mathbb R$, and
$x\in K$, we define the oscillation of $f$ along $D$ at $x$ by
\[
\operatorname{osc}_D f(x)
=
\inf_{V\in\mathcal V(x)}
\operatorname{diam}\bigl(f(V\cap D)\bigr),
\]
where $\mathcal V(x)$ denotes the family of open neighborhoods of $x$.

\begin{definition}
Let $\mathfrak I$ be a testing ideal and let $K$ be an
$\mathfrak I$-compatible compact Hausdorff space. We say that $K$ has
the \emph{$\Delta_{\mathfrak I}$-extension property} if there exists a
dense subset $D\subset K$ such that
\[
\operatorname{osc}_D(f)(x)=0
\]
for every $f\in\mathscr C_{\mathfrak I}(K)$ and every $x\in K$.
Such a set $D$ is called a
\emph{$\Delta_{\mathfrak I}$-extension set} for $K$.
\end{definition}

\begin{proposition}\label{Prop:Complementation}
Let $\mathfrak I$ be a testing ideal and let $K$ be an
$\mathfrak I$-compatible compact Hausdorff space with the
$\Delta_{\mathfrak I}$-extension property. Then $C(K)$ is
$1$-complemented in $\Delta_{\mathfrak I}(C(K))$.
More precisely, if $D$ is a $\Delta_{\mathfrak I}$-extension set for
$K$, then $P_D:\mathscr C_{\mathfrak I}(K)\to C(K)$, given by
\[
P_D(f)=\widetilde{f|_D},
\]
is a contractive lattice projection.
\end{proposition}

\begin{proof}
For every $f\in\mathscr C_{\mathfrak I}(K)$, the condition
$\operatorname{osc}_D(f)(x)=0$ for every $x\in K$ implies that
$f|_D$ has a unique continuous extension $\widetilde{f|_D}$ to $K$.
Thus $P_D$ is well defined. The uniqueness of the extension shows that
$P_D$ is linear, and
\[
\|P_D(f)\|
=
\sup_{x\in K}|\widetilde{f|_D}(x)|
=
\sup_{x\in D}|f(x)|
\leq
\|f\|_\infty.
\]
Moreover, $P_D(f)=f$ for every $f\in C(K)$, so $P_D$ is a contractive
projection onto $C(K)$. Finally,
\[
P_D(|f|)
=
\widetilde{\,|f||_D\,}
=
\bigl|\widetilde{f|_D}\bigr|
=
|P_D(f)|,
\]
and hence $P_D$ is a lattice homomorphism. Transporting $P_D$ through
the canonical isometric isomorphism $\Gamma:\Delta_{\mathfrak I}(C(K))\to \mathscr C_{\mathfrak I}(K)$
gives a norm-one projection from $\Delta_{\mathfrak I}(C(K))$ onto
$C(K)$.
\end{proof}

\begin{theorem}\label{Thm:ComplementationC0}
Let $\mathfrak I$ be a testing ideal and let $K$ be an
$\mathfrak I$-compatible compact Hausdorff space with the
$\Delta_{\mathfrak I}$-extension property. Let $L$ be a
$\Delta_{\mathfrak I}$-spectrum of $K$. Then there exists an open
locally compact subspace $U\subset L$ such that
\[
\mathcal S_{\mathfrak I}(C(K))
\sim
C_0(U).
\]
\end{theorem}
\begin{proof}
By Proposition~\ref{Prop:Complementation}, there exists a contractive
lattice projection $P:\Delta_{\mathfrak I}(C(K))\to C(K)$.
Therefore, $\ker P$ is a closed ideal of
$\Delta_{\mathfrak I}(C(K))$.

Let $J:\Delta_{\mathfrak I}(C(K))\to C(L)$ be an isometric lattice isomorphism witnessing that $L$ is a
$\Delta_{\mathfrak I}$-spectrum of $K$. Then $J(\ker P)$ is a closed
ideal of $C(L)$. Hence there exists a closed subset $F\subset L$ such
that
\[
J(\ker P)=\{h\in C(L):h|_F=0\}.
\]
Set $U=L\setminus F$. Then $U$ is an open locally compact Hausdorff
space and
\[
\ker P\cong J(\ker P)\cong C_0(U).
\]
Since $P$ is a projection onto $C(K)$, we have
\[
\mathcal S_{\mathfrak I}(C(K))
=
\Delta_{\mathfrak I}(C(K))/C(K)
\sim
\ker P
\sim
C_0(U).
\]
\end{proof}

Motivated by the approach used in the proof of
\cite[Lemma~4.4]{Korpalski}, we introduce the following abstract
sufficient condition for the
$\Delta_{\mathfrak I}$-extension property. Recall that
\[
\mathcal I_{\mathfrak I}(K)
=
\{A\subset K:\{\delta_a:a\in A\}\in
\mathfrak I(C(K)^*)\}.
\]

\begin{definition}\label{Def:Double}
Let $\mathfrak I$ be a testing ideal and let $K$ be an
$\mathfrak I$-compatible compact Hausdorff space. A subset $D\subset K$
is said to satisfy the
\emph{$\mathfrak I$-small double closure property} if, for every
$A,B\subset D$ and every $x\in\overline A\cap\overline B$, there exist
subsets $A_0\subset A$ and $B_0\subset B$ such that $A_0\cup B_0\in\mathcal I_{\mathfrak I}(K)$ 
and $\overline{A_0}\cap\overline{B_0}\neq\varnothing$.
\end{definition}

\begin{lemma}\label{Lem:DoubleExtension}
Let $\mathfrak I$ be a testing ideal and let $K$ be an
$\mathfrak I$-compatible compact Hausdorff space. If there exists a
dense subset $D\subset K$ satisfying the $\mathfrak I$-small double
closure property, then $D$ is a $\Delta_{\mathfrak I}$-extension set
for $K$. In particular, $K$ has the
$\Delta_{\mathfrak I}$-extension property.
\end{lemma}
\begin{proof}
Suppose, towards a contradiction, that there exist
$f\in\mathscr C_{\mathfrak I}(K)$ and $x\in K$ such that $\operatorname{osc}_D(f)(x)>0$.
Then there exist real numbers $r<s$ such that every neighborhood of
$x$ intersects both sets $A=\{y\in D:f(y)<r\}$ and  $B=\{y\in D:f(y)>s\}$.
Consequently, $x\in\overline A\cap\overline B$. By the
$\mathfrak I$-small double closure property, there exist
$A_0\subset A$ and $B_0\subset B$ such that $A_0\cup B_0\in\mathcal I_{\mathfrak I}(K)$ and $\overline{A_0}\cap\overline{B_0}\neq\varnothing$.
Since $f\in\mathscr C_{\mathfrak I}(K)$, we have $f|_{\overline{A_0\cup B_0}}
\in C(\overline{A_0\cup B_0})$. If we choose $y\in\overline{A_0}\cap\overline{B_0}$, we obtain $s\leq f(y)\leq r$ which is a contradiction. Therefore, $\operatorname{osc}_D(f)(x)=0$ for every $f\in\mathscr C_{\mathfrak I}(K)$ and every $x\in K$. Thus $D$ is a $\Delta_{\mathfrak I}$-extension set for $K$.
\end{proof}

We conclude the section by iterating the symmetric tensor-product
formula under suitable complementability assumptions. This yields an
explicit decomposition of the derivative of
$C(K_1\times\cdots\times K_n)$ in terms of the derivatives of the
individual factors.

\begin{theorem}\label{Thm:MultipleDerivatives}
Let $\mathfrak I$ be a testing ideal and let $K_1,\ldots,K_n$ be
$\mathfrak I$-compatible compact Hausdorff spaces. Assume that either
all the spaces $K_i$ are scattered or all of them are finite products
of compact lines. Suppose that $C(K_i)$ is complemented in
$\Delta_{\mathfrak I}(C(K_i))$ for every $1\leq i\leq n$. Then
\[
\mathcal S_{\mathfrak I}\!\big(
C(K_1\times\cdots\times K_n)
\big)
\sim
\bigoplus_{\varnothing\neq A\subset\{1,\ldots,n\}}
\big(
\widehat{\bigotimes}_{i\in A}
\mathcal S_{\mathfrak I}(C(K_i))
\;\widehat{\otimes}_{\varepsilon}\;
\widehat{\bigotimes}_{i\notin A}
C(K_i)
\big).
\]
\end{theorem}
\begin{proof}
We argue by induction on $n$. The case $n=1$ is trivial, while the case
$n=2$ is precisely Theorem~\ref{Thm:ProductRule}. Indeed, under the
present assumptions, $C(K_1)$ and $C(K_2)$ satisfy the hypotheses of
Theorem~\ref{Thm:TensorEmbeddingIntersection}.

Assume that the formula holds for some $n\geq2$ and set
\[
E=
\widehat{\bigotimes}_{i=1}^{n}C(K_i)
\cong
C(K_1\times\cdots\times K_n).
\]
By the induction hypothesis,
\[
\mathcal S_{\mathfrak I}(E)
\sim
\bigoplus_{\varnothing\neq A\subset\{1,\ldots,n\}}
\big(
\widehat{\bigotimes}_{i\in A}
\mathcal S_{\mathfrak I}(C(K_i))
\;\widehat{\otimes}_{\varepsilon}\;
\widehat{\bigotimes}_{i\notin A}
C(K_i)
\big).
\]

The compact space $K_1\times\cdots\times K_n$ is scattered if all the
spaces $K_i$ are scattered, and it is a finite product of compact lines
if all the spaces $K_i$ are finite products of compact lines. Hence the
results established above show that $E$ and $C(K_{n+1})$ satisfy the
hypotheses of
Theorem~\ref{Thm:TensorEmbeddingIntersection}.

We also claim that $E$ is complemented in
$\Delta_{\mathfrak I}(E)$. Indeed, for every $1\leq i\leq n$, let $P_i:\Delta_{\mathfrak I}(C(K_i))\to C(K_i)$
be a bounded projection. By repeated application of Theorem~\ref{Thm:TensorEmbeddingIntersection}, there is a canonical
isometric isomorphism $J:\Delta_{\mathfrak I}(E)\to\widehat{\bigotimes}_{i=1}^{n}\Delta_{\mathfrak I}(C(K_i))$
whose restriction to $E$ is the identity. Therefore, $P=\big(\widehat{\bigotimes}_{i=1}^{n}P_i\big)\circ J: \Delta_{\mathfrak I}(E)\to E$ is a bounded projection onto $E$.

Applying Theorem~\ref{Thm:ProductRule}, we obtain
\[\mathcal S_{\mathfrak I}
\bigl(E\widehat{\otimes}_\varepsilon C(K_{n+1})\bigr)\sim\bigl(\mathcal S_{\mathfrak I}(E)
\widehat{\otimes}_\varepsilon C(K_{n+1})
\bigr)\oplus\bigl(E\widehat{\otimes}_\varepsilon\mathcal S_{\mathfrak I}(C(K_{n+1}))\bigr)\oplus\bigl(\mathcal S_{\mathfrak I}(E)\widehat{\otimes}_\varepsilon\mathcal S_{\mathfrak I}(C(K_{n+1}))\bigr).
\]
Substituting the induction hypothesis and collecting the terms yields the stated
formula for $n+1$ factors.
\end{proof}


\section{Applications}
\label{Sec:Applications}

The derivative construction assigns to a Banach space $X$ the quotient
\[
\mathcal S_\Delta(X)=\Delta(X)/X,
\]
which measures the failure of $X$ to be a $\Delta$-Mazur space. For spaces of the form $C(K)$, under suitable compatibility and
extension assumptions, the study of the ideal-generated derivative
reduces to the analysis of the kernel of a natural projection. In particular,
Theorem~\ref{Thm:ComplementationC0} shows that it can be
represented as a space of the form $C_0(U)$, where $U$ is an open
locally compact subspace of a $\Delta_{\mathfrak I}$-spectrum of $K$.

In this section, we employ the results developed in the previous
sections to obtain applications at two different levels of generality.
We first derive consequences for arbitrary bidual assignments
satisfying the Semadeni condition, concerning injective tensor
products and direct sums. We then specialize the theory to the testing
ideals $\mathfrak I_\kappa$ and $\mathfrak I_{<\kappa}$ introduced in
Remark~\ref{Rem:ExamplesAssignments}. The ideal-generated assignment
associated with $\mathfrak I_\kappa$ coincides with the Korpalski
assignment $\Delta_\kappa$, and the corresponding derivative is the
Semadeni--Pe{\l}czy\'nski derivative. In this setting, explicit computations 
can be obtained for compact trees and, under suitable additional assumptions, for finite products
of compact lines.

Our first applications concern the behavior of bidual assignments
and their associated derivatives on spaces of continuous
vector-valued functions.

\begin{corollary}\label{Cor:ScatteredDelta}
Let $K$ be a scattered compact Hausdorff space, let $Y$ be a Banach
space, and let $\Delta$ be a bidual assignment satisfying the Semadeni
condition. If $C(K)$ is a $\Delta$-Mazur space, then
\[
\Delta(C(K,Y))
\cong
C\bigl(K,\Delta(Y)\bigr)\quad \text{ and }
\quad \mathcal S_\Delta(C(K,Y))
\cong
C\bigl(K,\mathcal S_\Delta(Y)\bigr)
\]
isometrically. In particular, if $Y$ is also a $\Delta$-Mazur space,
then $C(K,Y)$ is a $\Delta$-Mazur space.
\end{corollary}

\begin{proof}
By Corollary~\ref{Cor:ClassesGP}, $C(K)$ has the
Gelfand--Phillips property. Additionally, by
Proposition~\ref{Prop:CanonicalScattered},
$\widehat G(C(K)^*\otimes Y^*)$ is norm dense, and hence
weak$^*$-sequentially dense, in
$(C(K)\widehat\otimes_\varepsilon Y)^*$.

Since $C(K)$ also has the approximation property
\cite[Example~4.2]{Ryan}, all the hypotheses of
Theorem~\ref{Thm:DeltaInjectiveTensor} are satisfied. This theorem
gives $\Delta(C(K,Y))\cong C\bigl(K,\Delta(Y)\bigr)$
isometrically, while Theorem~\ref{Thm:CKTensorQuotient} gives $\mathcal S_\Delta(C(K,Y))
\cong C\bigl(K,\mathcal S_\Delta(Y)\bigr)$
isometrically. The final assertion follows immediately.
\end{proof}

\begin{remark}\label{Rem:CompactLineApp}
The conclusions of Corollary~\ref{Cor:ScatteredDelta} remain valid if
$K$ is a finite product of compact lines and either $Y$ is a
$\operatorname{WCG}$ Banach space or $Y=C(L)$ for some finite product
$L$ of compact lines. Indeed, the proof is analogous: the
Gelfand--Phillips property of $C(K)$ follows from
Corollary~\ref{Cor:ClassesGP}, while the required
weak$^*$-sequential density follows from
Proposition~\ref{Prop:CompactLinesWCG} in the first case and from
Proposition~\ref{Prop:CompactLinesAux} in the second one. Therefore,
whenever $C(K)$ is a $\Delta$-Mazur space,
\[
\Delta(C(K,Y))
\cong
C\bigl(K,\Delta(Y)\bigr)\quad \text{ and }\quad
\mathcal S_\Delta(C(K,Y))
\cong
C\bigl(K,\mathcal S_\Delta(Y)\bigr)
\]
isometrically.

In particular, if $K$ and $L$ are finite products of compact lines and
both $C(K)$ and $C(L)$ are $\Delta$-Mazur spaces, then
$C(K\times L)$ is a $\Delta$-Mazur space.
\end{remark}

Combining Theorems~\ref{Thm:IsometryDerivative} and
\ref{Thm:TensorQuotientDerivative}, we obtain the following formula
for derivatives of $\ell_p$-sums.

\begin{theorem}\label{Thm:ellpsumsofderivatives}
Let $\Delta$ be a bidual assignment satisfying the Semadeni condition,
let $(X_\alpha)_{\alpha<\xi}$ be a family of $\Delta$-Mazur Banach
spaces having the Gelfand--Phillips property, and let $K$ be a
scattered compact Hausdorff space such that $C(K)$ is complemented in
$\Delta(C(K))$ and $\Delta(C(K))$ has the approximation property. Assume that $p=0$ or $1<p<\infty$, where $\ell_0=c_0$. Then
\[
\mathcal S_\Delta\big(
\big(\bigoplus_{\alpha<\xi}C(K,X_\alpha)\big)_{\ell_p}
\big)
\sim
\big(
\bigoplus_{\alpha<\xi}
\mathcal S_\Delta(C(K))
\widehat\otimes_\varepsilon X_\alpha
\big)_{\ell_p}.
\]
Under the same hypotheses, the conclusion also holds for $p=1$
provided that $\xi$ is not a real-valued measurable cardinal.
\end{theorem}
\begin{proof}
For each $\alpha<\xi$, apply
Theorem~\ref{Thm:DeltaInjectiveTensor} with
$X=X_\alpha$ and $Y=C(K)$. Indeed, $X_\alpha$ is a
$\Delta$-Mazur space and has the Gelfand--Phillips property. Moreover,
by Proposition~\ref{Prop:CanonicalScattered},
$\widehat G(X_\alpha^*\otimes C(K)^*)$ is norm dense, and hence
weak$^*$-sequentially dense, in
$(X_\alpha\widehat\otimes_\varepsilon C(K))^*$.
Since $\Delta(C(K))$ has the approximation property, and $C(K)$ is complemented in $\Delta(C(K))$,
Theorem~\ref{Thm:TensorQuotientDerivative} yields
\[
\mathcal S_\Delta(C(K,X_\alpha))
\sim
X_\alpha\widehat\otimes_\varepsilon
\mathcal S_\Delta(C(K)).
\]
The distortions of these isomorphisms  are uniformly bounded, since they are induced by a fixed bounded projection from
$\Delta(C(K))$ onto $C(K)$.

Set $Z=\big(\bigoplus_{\alpha<\xi}C(K,X_\alpha)\big)_{\ell_p}$. By the results from Section~\ref{SeC:BidualAssignments}, we have 
\[\Delta(Z)=\big(\bigoplus_{\alpha<\xi}\Delta(C(K,X_\alpha))\big)_{\ell_p}.\]
Here, when $p=1$, we use the assumption that $\xi$ is not a
real-valued measurable cardinal. Therefore, combining
Theorem~\ref{Thm:IsometryDerivative} with the uniformly
bounded coordinatewise isomorphisms above gives
\[
\mathcal S_\Delta(Z)
\cong
\big(
\bigoplus_{\alpha<\xi}
\mathcal S_\Delta(C(K,X_\alpha))
\big)_{\ell_p}\sim
\big(
\bigoplus_{\alpha<\xi}
\mathcal S_\Delta(C(K))
\widehat\otimes_\varepsilon X_\alpha
\big)_{\ell_p}.
\]
\end{proof}

When the spaces $X_\alpha$ are themselves spaces of continuous
functions, Theorem~\ref{Thm:CKTensorQuotient} allows us to remove the
complementation and approximation-property assumptions imposed on
$\Delta(C(K))$ in the preceding theorem. More precisely, we obtain the
following isometric version.

\begin{corollary}
Let $\Delta$ be a bidual assignment satisfying the Semadeni condition,
let $K$ be either a scattered compact Hausdorff space or a finite
product of compact lines, and let $(M_\alpha)_{\alpha<\xi}$ be a family
of finite products of compact lines such that each $C(M_\alpha)$ is a
$\Delta$-Mazur space. Assume that $p=0$ or $1<p<\infty$, where
$\ell_0=c_0$. Then
\[
\mathcal S_\Delta\big(
\big(\bigoplus_{\alpha<\xi}C(K\times M_\alpha)\big)_{\ell_p}
\big)
\cong
\big(
\bigoplus_{\alpha<\xi}
C\big(M_\alpha,\mathcal S_\Delta(C(K))\big)
\big)_{\ell_p}
\]
isometrically. The same conclusion holds for $p=1$ provided that
$\xi$ is not a real-valued measurable cardinal.
\end{corollary}

\begin{proof}
For each $\alpha<\xi$, the space $C(M_\alpha)$ has the
Gelfand--Phillips property by Corollary~\ref{Cor:ClassesGP} and has the
approximation property. Moreover, the required weak$^*$-sequential
density follows from Proposition~\ref{Prop:CanonicalScattered} when
$K$ is scattered and from Proposition~\ref{Prop:CompactLinesAux} when
$K$ is a finite product of compact lines. Since $C(M_\alpha)$ is a
$\Delta$-Mazur space, Theorem~\ref{Thm:CKTensorQuotient} gives
\[
\mathcal S_\Delta(C(K\times M_\alpha))
\cong
C\bigl(M_\alpha,\mathcal S_\Delta(C(K))\bigr)
\]
isometrically.

The conclusion now follows from
Theorem~\ref{Thm:IsometryDerivative} and the direct-sum identifications
from Section~\ref{SeC:BidualAssignments}. In the case $p=1$, we use the
additional assumption that $\xi$ is not a real-valued measurable
cardinal.
\end{proof}

We now specialize the general theory to the testing ideals
$\mathfrak I_\kappa$ and $\mathfrak I_{<\kappa}$. We consider $\mathfrak I_\kappa$ for arbitrary infinite cardinals
$\kappa$, whereas $\mathfrak I_{<\kappa}$ will be used only when
$\kappa$ is an uncountable regular cardinal. Under these assumptions,
Remark~\ref{Rem:ExamplesCompatible} shows that scattered compact
Hausdorff spaces and finite products of compact lines are compatible
with the respective testing ideals. Consequently, the results of
Section~\ref{Sec:DerivativesC(K,Y)} apply to both classes.

For consistency of notation, we write
\[
\mathcal S_\kappa(X)
=
\mathcal S_{\mathfrak I_\kappa}(X)
=
\Delta_\kappa(X)/X.
\]
Likewise, for every uncountable regular cardinal $\kappa$, we write
\[
\mathcal S_{<\kappa}(X)
=
\mathcal S_{\mathfrak I_{<\kappa}}(X)
=
\Delta_{<\kappa}(X)/X.
\]

We begin with compact lines. Korpalski's description of
$\mathcal S_\kappa(C(L))$ leads to an explicit computation of the
Semadeni--Pe{\l}czy\'nski derivative of spaces of the form
$C(K\times L)$. Combining Korpalski's description with
Corollary~\ref{Cor:ScatteredDelta} and
Remark~\ref{Rem:CompactLineApp}, we obtain the following corollary.

\begin{corollary}
Let $\kappa$ be an infinite cardinal, let $L$ be a compact line, and
let $K$ be either a scattered compact Hausdorff space or a finite
product of compact lines. Assume that $C(K)$ is $\Delta_\kappa$-Mazur space.
Then
\[
\mathcal S_\kappa(C(K\times L))
\sim
c_0(\Gamma,C(K)),
\]
where $\Gamma=L^\uparrow$ is the discrete set associated with $L$ in
\cite[Lemma~4.5]{Korpalski}.
\end{corollary}
\begin{proof}
Corollary~\ref{Cor:ScatteredDelta}, or
Remark~\ref{Rem:CompactLineApp} when $K$ is a finite product of compact
lines, gives the isometric identification
\[
\mathcal S_\kappa(C(K\times L))
\cong
C\bigl(K,\mathcal S_\kappa(C(L))\bigr).
\]
By \cite[Lemma~4.5]{Korpalski},
$\mathcal S_\kappa(C(L))\sim c_0(\Gamma)$, where
$\Gamma=L^\uparrow$. Therefore,
\[
\mathcal S_\kappa(C(K\times L))
\sim
C(K)\widehat\otimes_\varepsilon c_0(\Gamma)
\sim
c_0(\Gamma,C(K)).
\]
\end{proof}

We now turn to compact trees, which form the main class of examples
considered in this section and include ordinal intervals as a
particularly important case. For this class, we obtain an explicit
description of $\mathcal S_{<\kappa}(C(K))$.

\begin{theorem}\label{Thm:CompactTrees}
Let $K$ be a compact tree and let $\kappa$ be an uncountable regular cardinal.
Define $\Lambda_\kappa(K)=\{x\in K:\chi(K,x)\geq \kappa\}$. Then
\[
\mathcal S_{<\kappa}(C(K))
\sim
c_0(\Lambda_\kappa(K)).
\]
\end{theorem}
\begin{proof}
We set $D=K\setminus\Lambda_\kappa(K)$. Since compact trees are
scattered, $K$ is $\mathfrak I_{<\kappa}$-compatible by
Remark~\ref{Rem:ExamplesCompatible}. Clearly, $D$ is dense in $K$,
since it contains all isolated points of the tree. 
We prove that $D$ satisfies the $\mathfrak{I}_{<\kappa}$-small double closure property for $K$, see Definition \ref{Def:Double}.

Indeed, let $A,B\subset D$ be arbitrary subsets and assume that
$x\in \overline A\cap \overline B$. If $x\in D$, fix a local basis
$\mathcal B_x$ at $x$ with $|\mathcal B_x|<\kappa$. For each
$V\in \mathcal B_x$, choose points $a_V\in V\cap A$ and $b_V\in V\cap B$. Then
$A_0=\{a_V:V\in\mathcal B_x\}$ and $B_0=\{b_V:V\in\mathcal B_x\}$ satisfy $|A_0\cup B_0|<\kappa$ and
$x\in\overline{A_0}\cap\overline{B_0}$.

Assume now that $x\notin D$. Since $A,B\subset D$, we have $x\notin A\cup B$.
Thus $A\cap \operatorname{pred}(x)$ and $B\cap \operatorname{pred}(x)$ are
cofinal in $\operatorname{pred}(x)$. Hence we may construct, by induction,
sequences $(a_n)_n$ in $A\cap \operatorname{pred}(x)$ and $(b_n)_n$ in
$B\cap \operatorname{pred}(x)$ such that
\[
a_1<b_1<a_2<b_2<\cdots .
\]
Let $y=\sup_{n\in\mathbb N}a_n=\sup_{n\in\mathbb N}b_n$, which exists by compactness of the tree. Define $A_0=\{a_n:n\in\mathbb N\}$ and  $B_0=\{b_n:n\in\mathbb N\}$. Then $|A_0\cup B_0|=\omega<\kappa$ and $y\in \overline{A_0}\cap \overline{B_0}$. Thus $D$ satisfies the $\mathfrak I_{<\kappa}$-small double closure property.

By Lemma~\ref{Lem:DoubleExtension}, $D$ is a
$\Delta_{\mathfrak I_{<\kappa}}$-extension set for $K$. Under the
canonical identification of $\Delta_{<\kappa}(C(K))$ with
$\mathscr C_{\mathfrak I_{<\kappa}}(K)$,
Proposition~\ref{Prop:Complementation} yields a contractive lattice
projection $P:\Delta_{<\kappa}(C(K))\to C(K)$, given by $P(f)=\widetilde{f|_D}$, where 
$\widetilde{f|_D}$ denotes the unique continuous extension of
$f|_D$ to $K$. Moreover, by Theorem~\ref{Thm:ComplementationC0}, if $L$ is a
$\Delta_{<\kappa}$-spectrum of $K$, then there exists an open locally
compact subspace $U\subset L$ such that
\[
\mathcal S_{<\kappa}(C(K))\sim\ker P\sim C_0(U).
\]

We now prove that $C_0(U)\sim c_0(K\setminus D)$. Consider the operator
$R:\Delta_{<\kappa}(C(K))\to \Delta_{<\kappa}(C(K))$ defined by $R(h)=h-P(h)$. Since 
$P$ is a projection, $R$ is also a projection and $\operatorname{Im}(R)=\ker P.$. Let $h \in  \Delta_{<\kappa}(C(K))$ be arbitrary and let 
\[
E_\varepsilon(h)=\{x\in K:|R(h)(x)|\geq \varepsilon\}.
\]
We claim that $E_\varepsilon(h)$ is finite. Otherwise, there exists an infinite countable subset $W\subset E_\varepsilon(h)$. Let $w$ be an accumulation point of $W$. Since $|W|=\omega< \kappa$ we obtain that $R(h)|_{\overline{W}}$ is continuous (Proposition~\ref{Prop:Fecho} and Remark~\ref{Rem:Fecho}) and therefore, $|R(h)(w)|\geq \varepsilon$. If $\operatorname{pred}(w)\cap W$ is cofinal in $\operatorname{pred}(w)$, then $\operatorname{cf}(\operatorname{ht}(w))\leq\omega$. Since $\chi(K,w)=\operatorname{cf}(\operatorname{ht}(w))$, we obtain $\chi(K,w)<\kappa$, and thus $w\in D$. Consequently $R(h)(w)=0$ which is a contradiction. Otherwise, there exists $v<w$ such that $(v,w]\cap W\subset \{w\}$, this implies that $w$ is not an accumulation point of $W$, again a contradiction. This establishes our claim.

Define $\Phi:\ker P \to c_0(K\setminus D)$ by the formula
\[\Phi(h)=h|_{K\setminus D}.\]
The preceding claim shows that $\Phi$ is well defined. Moreover, if
$h\in\ker P$, then $h|_D=P(h)|_D=0$. Therefore,
\[\|h\|=\sup_{x\in K\setminus D}|h(x)|=\|\Phi(h)\|,\]
and hence $\Phi$ is a linear isometric embedding.

It remains to prove that $\Phi$ is surjective. Let $x\in K\setminus D$ be arbitrary and let $\Omega\subset K$ satisfy $|\Omega|<\kappa$ and $x\in \overline{\Omega}$. Since $\chi(K,x)\geq \kappa$ we have that $\Omega\cap\operatorname{pred}(x)$ is not cofinal in $\operatorname{pred}(x)$. Hence there exists $y<x$ such that $(y,x]\cap \Omega= \{x\}$. Therefore, $x$ is isolated in $\overline\Omega$, and hence $\mathbf 1_{\{x\}}|_{\overline{\Omega}}\in C(\overline{\Omega})$. By Proposition~\ref{Prop:Fecho} we deduce that $\mathbf 1_{\{x\}}\in \Delta_{<\kappa}(C(K))$ for each $x\in K\setminus D$. Thus
\[c_0(K\setminus D)\cong \overline{\operatorname{span}{\{\mathbf 1_{\{x\}}:x\in K\setminus D\}}} \subset \ker P\]
and this completes the proof that $\Phi$ is an isomorphism.
\end{proof}

\begin{remark}\label{Rem:Kislyakov}
Motivated by \cite[Corollary~4.1]{Kislyakov}, in the particular case
$K=[0,\alpha]$ we write
\[
\Lambda_\kappa(\alpha)
=
\Lambda_\kappa([0,\alpha]).
\]
For ordinal intervals, this notation coincides with the set considered by
Kislyakov, namely the set of points that are not accumulation points of any
subset of cardinality strictly smaller than $\kappa$.
\end{remark}

For every Banach space $X$, we set
$\mathcal S_{<\kappa}^{(0)}(X)=X$ and define recursively, for each
$n\in\mathbb N$,
\[\mathcal S_{<\kappa}^{(n+1)}(X)=\mathcal S_{<\kappa}\bigl(\mathcal S_{<\kappa}^{(n)}(X)\bigr).\]

The iterated derivatives $\mathcal S_\kappa^{(n)}$ are defined analogously.

By Theorem~\ref{Thm:CompactTrees}, whenever $K$ is a compact tree, $\mathcal S_{<\kappa}(C(K))\sim
c_0(\Lambda_\kappa(K))$. In particular, $\mathcal S_{<\kappa}(C(K))$ is a Mazur space and therefore
\[
\mathcal S_{<\kappa}^{(2)}(C(K))
=
\{0\}.
\]

The next lemma exploits this observation together with the tensor-product formula,
Theorem~\ref{Thm:ProductRule} (and Theorem~\ref{Thm:MultipleDerivatives})
to obtain an explicit description of the iterated
$\mathcal S_{<\kappa}$-derivatives of finite products of compact trees.

\begin{lemma}\label{Lem:DerivativeNtimes}
Let $K_1,\ldots,K_n$ be compact trees and let $\kappa$ be an uncountable regular cardinal, and let $X$ be a Mazur space having the Gelfand--Phillips property. Then
\[
\mathcal{S}_{<\kappa}^{(n)}(C(K_1\times \cdots \times K_n,X)) \sim
c_0\bigl(\{1,\ldots,n!\}\times \Lambda_\kappa(K_1)\times \cdots \times\Lambda_\kappa(K_n),X\bigr).
\]
In particular, if $\Lambda_\kappa(K_1)\times \cdots \times\Lambda_\kappa(K_n)$ is infinite, then
\[
\mathcal{S}_{<\kappa}^{(n)}(C(K_1\times \cdots \times K_n,X))
\sim c_0(\Lambda_\kappa(K_1)\times \cdots \times\Lambda_\kappa(K_n),X).
\]
\end{lemma}
\begin{proof}
By Theorem~\ref{Thm:TensorQuotientDerivative},
\[
\mathcal S_{<\kappa}^{(n)}
(C(K_1\times\cdots\times K_n,X))
\sim
\mathcal S_{<\kappa}^{(n)}
(C(K_1\times\cdots\times K_n))
\widehat\otimes_\varepsilon X.
\]
Therefore it is enough to prove the result in the scalar-valued case. For each $1\leq i\leq n$, put $E_i=C(K_i)$ and $D_i=\mathcal{S}_{<\kappa}(E_i)$. By the preceding discussion, $D_i\sim c_0(\Lambda_\kappa(K_i))$. And since each $D_i$ is a Mazur space, we have $\mathcal{S}_{<\kappa}(D_i)=0$.

Since $C(K_1\times \cdots \times K_n) \sim E_1\widehat{\otimes}_{\varepsilon}\cdots\widehat{\otimes}_{\varepsilon}E_n$, we compute the iterated derivative of the latter space. The key observation is the following inductive description: for every $0\leq r\leq n$,
\[
\mathcal{S}_{<\kappa}^{(r)}
(E_1\widehat{\otimes}_{\varepsilon}\cdots
\widehat{\otimes}_{\varepsilon}E_n)
\sim
\bigoplus_{\mathcal P}
\widehat{\bigotimes}_{i=1}^n F_i^{\mathcal P},
\]
where the sum is taken over all ordered partitions $\mathcal P=(A_1,\ldots,A_r)$ of a subset of $\{1,\ldots,n\}$ into $r$ nonempty pairwise disjoint blocks, and
\[
F_i^{\mathcal P}=
\begin{cases}
D_i, & \text{if } i\in A_1\cup\cdots\cup A_r,\\
E_i, & \text{otherwise}.
\end{cases}
\]

For $r=0$, the formula is clear. Assume that it holds for some
$r<n$. Applying Theorem~\ref{Thm:MultipleDerivatives} to each summand in the
decomposition and using that $\mathcal S_{<\kappa}$ commutes with finite
direct sums, we obtain a new finite direct sum.

Let $(A_1,\ldots,A_r)$ be an ordered partition indexing one of the summands. Since $\mathcal S_{<\kappa}(D_i)=0$ for every $i$, the only nonzero terms arising
from this summand are obtained by choosing a nonempty subset of the indices
not belonging to
$
A_1\cup\cdots\cup A_r
$
and replacing the corresponding factors $E_i$ by $D_i$.
Equivalently, one obtains a new ordered partition
$
(A_1,\ldots,A_r,A_{r+1})
$,
where $A_{r+1}$ is a nonempty subset of the remaining indices.

Consequently, the summands in the new decomposition are naturally indexed by
the ordered partitions with $r+1$ nonempty blocks. Conversely, every such
ordered partition arises uniquely in this way. Therefore the formula holds
for $r+1$.

At the final stage $r=n$, the only possible subset of $\{1,\ldots,n\}$ that can be partitioned into $n$ nonempty blocks is the whole set, and each block must be a singleton. Hence the ordered partitions are exactly the $n!$ permutations of $\{1,\ldots,n\}$. Therefore
\[
\mathcal{S}_{<\kappa}^{(n)}
(E_1\widehat{\otimes}_{\varepsilon}\cdots
\widehat{\otimes}_{\varepsilon}E_n)
\sim
\bigoplus_{m=1}^{n!}
D_1\widehat{\otimes}_{\varepsilon}\cdots
\widehat{\otimes}_{\varepsilon}D_n.
\]
Since $D_1\widehat{\otimes}_{\varepsilon}\cdots\widehat{\otimes}_{\varepsilon}D_n\sim
c_0(\Lambda_{\kappa}(K_1)\times\cdots\times
\Lambda_{\kappa}(K_n))$, we obtain
\[
\mathcal{S}_{<\kappa}^{(n)}(C(K_1\times \cdots \times K_n))
\sim
c_0\bigl(\{1,\ldots,n!\}\times
\Lambda_{\kappa}(K_1)\times\cdots\times
\Lambda_{\kappa}(K_n)\bigr).
\]
The final assertion follows from the fact that
$c_0(\{1,\ldots,n!\}\times \Gamma)\sim c_0(\Gamma)$ whenever $\Gamma$ is infinite.
\end{proof}

\begin{remark}\label{Rem:KorpalskiTheorem}
By combining Theorem~\ref{Thm:MultipleDerivatives} with
Theorem~\ref{Thm:TensorQuotientDerivative} and
Proposition~\ref{Prop:CompactLinesWCG}, one can obtain an analogue of
Lemma~\ref{Lem:DerivativeNtimes} for compact lines under an additional
complementability assumption. More precisely, let $\kappa$ be an
infinite cardinal, let $n\in\mathbb N$, and let
$K_1,\ldots,K_n$ be compact lines such that
$C(K_i)$ is complemented in $\Delta_\kappa(C(K_i))$ for every
$1\leq i\leq n$. Let $X$ be a $\operatorname{WCG}$ Banach space with the
Gelfand--Phillips property.
Then
\[
\mathcal S_\kappa^{(n)}
\bigl(C(K_1\times\cdots\times K_n,X)\bigr)
\sim
c_0\bigl(\{1,\ldots,n!\}\times
\Gamma_1\times\cdots\times\Gamma_n,X\bigr),
\]
where, for each $1\leq i\leq n$, $\mathcal S_\kappa(C(K_i))\sim c_0(\Gamma_i)$; see \cite[Lemma~4.5]{Korpalski}.
\end{remark}

The following elementary lemma is essentially contained in
\cite{Kislyakov}. We include a proof for completeness.

\begin{lemma}\label{Lem:CardKis}
Let $\lambda$ be an uncountable regular initial ordinal. If $\alpha=\lambda\alpha'+\delta$, with $\alpha'\leq\lambda$ and $\delta<\lambda$,
then $|\Lambda_\lambda(\alpha)|=|\alpha'|$.
\end{lemma}
\begin{proof}
We notice that the set $\{\lambda\theta:0<\theta\leq \alpha'\}$ is contained in $\Lambda_\lambda(\alpha)$. On the other hand, if
$\xi\in \Lambda_\lambda(\alpha)$, then writing $\xi=\lambda \theta +\rho$, with $\rho<\lambda$, we deduce that
$\rho=0$ and $0<\theta\leq\alpha'$.Hence $\Lambda_\lambda(\alpha)=
\{\lambda\theta:0<\theta\leq\alpha'\}$. Therefore, $|\Lambda_\lambda(\alpha)|=|\alpha'|.$
\end{proof}

Combining the explicit description of the iterated derivatives obtained above
with Lemma~\ref{Lem:CardKis}, we are now ready to prove our main
classification theorem.

\begin{theorem}\label{Thm:MainSuper}
Let $\alpha$ and $\beta$ be uncountable ordinals with $\alpha\leq\beta$,
let $n\geq 1$, and let $X$ be a Banach space with the Mazur property and
the Gelfand--Phillips property, containing no subspace isomorphic to $c_0$.
Assume moreover that $X^p\not\sim X^q$ whenever $p$ and $q$ are distinct positive integers. Then
\[
C([0,\alpha]^n,X)\sim C([0,\beta]^n,X)\quad\text{if and only if}\quad C([0,\alpha])\sim C([0,\beta]).
\]
\end{theorem}
\begin{proof}
The converse implication is immediate. Thus, it remains to prove the direct implication. Since $C([0,\beta]^n)$ embeds into $C([0,\beta]^n,X)$ and $C([0,\beta]^n,X)\sim C([0,\alpha]^n,X)$, \cite[Theorem~1.3]{Candido2016} gives $|\beta|\le |\alpha|$. Since $\alpha\le\beta$, it follows that $|\alpha|=|\beta|$. Let $\kappa$ be an uncountable regular cardinal. By employing Lemma \ref{Lem:DerivativeNtimes} we obtain 
\[c_0\bigl(\{1,\ldots,n!\}\times
(\Lambda_\kappa(\alpha))^n,X\bigr)
\sim c_0\bigl(\{1,\ldots,n!\}\times
(\Lambda_\kappa(\beta))^n,X\bigr).
\]

Applying again \cite[Theorem~1.3]{Candido2016}, we obtain $|\Lambda_\kappa(\alpha)|=|\Lambda_\kappa(\beta)|$
whenever $\Lambda_\kappa(\alpha)$ and $\Lambda_\kappa(\beta)$ are infinite.
If they are finite, the same conclusion follows from the assumption that
$X^p\not\sim X^q$ whenever $p$ and $q$ are distinct positive integers.

If $n=1$, the desired conclusion follows directly from
\cite[Theorem~1.1]{galego4}, together with the assumption that
$X^p\not\sim X^q$ for distinct positive integers $p$ and $q$.
Thus, in the remainder of the proof we assume $n\ge 2$. Applying the inductive decomposition from the proof of
Lemma~\ref{Lem:DerivativeNtimes} at the stage $r=n-1$, the summands are indexed
by ordered partitions of subsets of $\{1,\ldots,n\}$ into $n-1$ nonempty
blocks. Such partitions are of two types: either they partition a subset of
cardinality $n-1$, producing the term
\[
c_0(A\times(\Lambda_\kappa(\alpha))^{n-1},C([0,\alpha],X)),
\]
where $A=\{\text{ordered partitions of an }(n-1)\text{-element subset into singletons}\}$
or they partition the whole set $\{1,\ldots,n\}$, in which case exactly one
block has cardinality two, producing the term
\[
c_0(B\times(\Lambda_\kappa(\alpha))^n,X),
\]
where $B=\{\text{ordered partitions of }\{1,\ldots,n\}\text{ into }n-1\text{ blocks}\}$. Simple combinatorics gives $|A|=n!$ and $|B|=\frac{n(n-1)}{2}(n-1)!$.

We can write 
\[\mathcal{S}_{<\kappa}^{(n-1)}
(C([0,\alpha]^n,X) \sim  c_0(A\times (\Lambda_\kappa(\alpha))^{n-1}, C([0,\alpha],X))\oplus c_0(B \times (\Lambda_\kappa(\alpha))^{n},X).\]

Since $|B\times(\Lambda_\kappa(\alpha))^n|\leq |\alpha|$, and since
$C([0,\alpha],X)$ contains a complemented copy of
$c_0(\Gamma,X)$ with $|\Gamma|=|\alpha|$, the second summand is absorbed by
the first one. Hence
\[
\mathcal{S}_{<\kappa}^{(n-1)}(C([0,\alpha]^n,X)
\sim
c_0(A\times(\Lambda_\kappa(\alpha))^{n-1},C([0,\alpha],X)).
\]

The same computation for $C([0,\beta]^n)$, gives  
\[c_0(A\times (\Lambda_\kappa(\alpha))^{n-1}, C([0,\alpha],X))\sim  c_0(A\times (\Lambda_\kappa(\beta))^{n-1}, C([0,\beta],X))\]
where, as shown above, $|\Lambda_\kappa(\alpha)|=|\Lambda_\kappa(\beta)|$. Hence, for a suitable set $\Gamma_\kappa$, the previous relation may be
rewritten, using the canonical identifications as
\begin{equation}\label{Rel:equa}
c_0(\Gamma_\kappa,X)^\alpha
\sim
c_0(\Gamma_\kappa,X)^\beta
\end{equation}
in the notation of \cite{galego2}.

By \cite[Theorem~1.1(4)]{galego2} and \eqref{Rel:equa}, there is no
uncountable regular ordinal $\lambda$ such that $\lambda\leq \alpha<\lambda^2\leq \beta$.
Let $\lambda$ be the initial ordinal satisfying $|\lambda|=|\alpha|=|\beta|$. Therefore, only two cases may occur.

First, suppose that $\lambda$ is singular, or that $\lambda$ is an uncountable regular ordinal with $\lambda^2\leq \alpha$. Then, by
\cite[Theorem~1.1(2)]{galego2} and \eqref{Rel:equa}, we obtain $\beta<\alpha^\omega$. Hence, \cite[Theorem~1]{Kislyakov} yields $C([0,\alpha])\sim C([0,\beta])$.

It remains to consider the case where $\lambda$ is an uncountable regular
ordinal and $\alpha,\beta\in[\lambda,\lambda^2]$. Write
\[
\alpha=\lambda\alpha'+\gamma,
\qquad
\beta=\lambda\beta'+\delta,
\]
where $\alpha',\beta'\leq\lambda$ and $\gamma,\delta<\lambda$.

Taking $\kappa=\lambda$ and using Lemma~\ref{Lem:CardKis}, we obtain
\[
|\alpha'|=|\Lambda_\lambda(\alpha)|=|\Lambda_\lambda(\beta)|=|\beta'|.
\]
Therefore, \cite[Theorem~2]{Kislyakov} implies that $C([0,\alpha])\sim C([0,\beta])$.
\end{proof}

\begin{remark}
The assumption that $X^p\not\sim X^q$ for distinct positive integers
$p$ and $q$ is needed to exclude finite multiplicity phenomena.
Indeed, if $X\sim X^2$, then $X\sim X^n$ for every positive integer $n$.
Consequently, $C([0,\omega_1],X)\sim C([0,\omega_1 n],X)$ for every positive integer $n$.
On the other hand, it follows from \cite{Se2} that
\[
C([0,\omega_1])\not\sim C([0,\omega_1 n])
\]
whenever $n\neq 1$.

There are, however, many Banach spaces satisfying the rigidity assumptions of
Theorem~\ref{Thm:MainSuper}. Let $H$ be the reflexive hereditarily
indecomposable Banach space constructed by Ferenczi in \cite{Ferenczi1997}.
For each $1<s<\infty$, set $X_s=\ell_s\oplus H$. The family $\{X_s:1<s<\infty\}$ consists of pairwise non-isomorphic reflexive Banach spaces. In particular, each $X_s$ is a Mazur space with the Gelfand--Phillips property and contains no isomorphic copy of $c_0$.

Moreover, \cite[Remark 4.1]{CandidoGalego2013} shows that, for every
$s\in(1,\infty)$ and every $p,q\in\mathbb N$, $X_s^p\sim X_s^q$ if and only if $p=q$.
\end{remark}

We close this section with the following open problem, which naturally extends the scope of Theorem~\ref{Thm:MainSuper}.

\begin{problem}
Classify, up to isomorphism, the spaces $C([0,\alpha_1]\times\cdots\times[0,\alpha_n])$,
where $\alpha_1,\ldots,\alpha_n$ are uncountable ordinals.
\end{problem}

The results of the present paper suggest that the iterated
$\mathcal S_{<\kappa}$-derivatives may provide useful invariants for this
problem. Indeed, the formulas from Lemma~\ref{Lem:DerivativeNtimes}, together
with the classification result of Theorem~\ref{Thm:MainSuper}, show that these
derivatives recover significant ordinal information from finite powers of
ordinal intervals.

\section*{Acknowledgements}

The author is grateful to Professor Vin\'icius Morelli Cortes (IME--USP) for
suggesting the main idea behind the proof of
Theorem~\ref{Thm:DeltaInjectiveTensor} and for generously allowing its
inclusion in this work. The author also thanks him for many helpful
discussions on tensor products, as well as for a course on the subject taught
in 2021, which proved valuable in the development of the present research.

The author is deeply grateful to the referee for the careful reading
of the manuscript and for the detailed comments, suggestions, and
proof outlines, which helped broaden the scope of several results,
correct a number of gaps and inaccuracies, and substantially improve
the final version. 

The author is also grateful to Maciej Korpalski for
a helpful observation concerning the complementability of $C(K)$ in
$\Delta_\kappa(C(K))$ for compact lines.

This research was supported by the Funda\c{c}\~ao de Amparo \`a Pesquisa do Estado de S\~ao Paulo (FAPESP), Grant No.~2023/12916-1.


\begin{thebibliography}{HD}

\normalsize
\baselineskip=17pt

\bibitem{Banach}
S. Banach,
\emph{Th\'eorie des op\'erations lin\'eaires},
Monografie Matematyczne,
Warsaw, 1932.

\bibitem{BourgainFremlinTalagrand}
J. Bourgain, D. H. Fremlin and M. Talagrand,
\emph{Pointwise compact sets of Baire-measurable functions},
Amer. J. Math. \textbf{100} (1978), no.~4, 845--886.

\bibitem{CandidoGalego2013}
L. Candido and E. M. Galego,
\emph{Embeddings of $C(K)$ spaces into $C(S,X)$ spaces with distortion strictly less than $3$},
Fund. Math. \textbf{220} (2013), 83--92.

\bibitem{Candido2016}
L. Candido,
\emph{On embeddings of $C_0(K)$ spaces into $C_0(L,X)$ spaces},
Studia Math. \textbf{233} (2016), no.~3, 1--6.

\bibitem{CandidoSquares}
L. Candido,
\emph{On the Semadeni derivative of Banach spaces $C(K,X)$},
Studia Math. \textbf{266} (2022), no.~2, 225--240.

\bibitem{Diestel}
J. Diestel,
\emph{Sequences and Series in Banach Spaces},
Springer, New York, 1984.

\bibitem{EGT}
K. P. Hart, J. Nagata and J. E. Vaughan (eds.),
\emph{Encyclopedia of General Topology},
Elsevier, Amsterdam, 2004.

\bibitem{Fabian}
M. Fabian, P. Habala, P. H\'ajek, V. Montesinos and V. Zizler,
\emph{Banach Space Theory: The Basis for Linear and Nonlinear Analysis},
CMS Books in Mathematics,
Springer, New York, 2011.

\bibitem{Ferenczi1997}
V. Ferenczi,
\emph{A uniformly convex hereditarily indecomposable Banach space},
Israel J. Math. \textbf{102} (1997), 199--225.

\bibitem{galego1}
E. M. Galego,
\emph{Banach spaces of continuous vector-valued functions of ordinals},
Proc. Edinb. Math. Soc. (2) \textbf{44} (2001), no.~1, 49--62.

\bibitem{galego2}
E. M. Galego,
\emph{Complete isomorphic classifications of some spaces of compact operators},
Proc. Amer. Math. Soc. \textbf{138} (2010), 725--736.

\bibitem{galego3}
E. M. Galego,
\emph{How to generate new Banach spaces non-isomorphic to their Cartesian squares},
Bull. Polish Acad. Sci. Math. \textbf{47} (1999), no.~1, 21--25.

\bibitem{galego4}
E.~M. Galego,
\newblock Spaces of compact operators on $C(2^m\oplus[0,\alpha])$ spaces,
\newblock {\em J. Math. Anal. Appl.} \textbf{370} (2010), no.~2, 406--414.

\bibitem{Gelfand1938}
I. M. Gelfand,
\emph{Abstrakte Funktionen und lineare Operatoren},
Mat. Sb. \textbf{46} (1938), 235--286.


\bibitem{GodefroyRosenthal}
G. Godefroy,
\emph{Compacts de Rosenthal},
Pacific J. Math.
91 (1980), no.~2, 293--306.

\bibitem{Jech}
T. Jech,
\emph{Set Theory},
Springer Monographs in Mathematics,
Springer, Berlin, 2003.

\bibitem{Kakutani}
S. Kakutani,
\emph{Concrete representation of abstract $(M)$-spaces},
Ann. of Math. (2) \textbf{42} (1941), 994--1024.

\bibitem{Kapp}
T. Kappeler,
\emph{Banach Spaces with the Condition of Mazur},
Math. Z. \textbf{191} (1986), 623--631.

\bibitem{Kislyakov}
S. V. Kislyakov,
\emph{Classification of spaces of continuous functions of ordinals},
Siberian Math. J. \textbf{16} (1975), no.~2, 226--231.

\bibitem{KorpalskiThesis}
M. Korpalski,
\emph{O izomorficznych i geometrycznych w\l asno\'sciach przestrzeni Banacha funkcji ci\c{a}g\l ych},
Ph.D. thesis,
University of Wroc\l aw, 2026.

\bibitem{Korpalski}
M. Korpalski,
\emph{Semadeni--Pe{\l}czy\'nski Derivative and Banach Spaces of Continuous Functions on Nonmetrizable Cubes},
arXiv:2502.16981 [math.FA], 2025.

\bibitem{Leung}
D. Leung,
\emph{On Banach spaces with Mazur's property},
Glasgow Math. J. \textbf{33} (1991), 51--54.

\bibitem{Mercourakis}
S. Mercourakis,
\emph{Some remarks on countably determined measures and uniform distribution of sequences}.
Monatsh. Math. \textbf{121}, 79--111 (1996)

\bibitem{Namioka}
I. Namioka,
\emph{Radon-Nikod\'ym compact spaces and fragmentability},
Mathematika \textbf{34} (1987), no.~2, 258--281.

\bibitem{Ostaszewski}
A.~J. Ostaszewski,
\emph{A characterization of compact, separable, ordered spaces},
J. London Math. Soc. (2) \textbf{7} (1974), no.~4, 758--760.

\bibitem{Ryan}
R. A. Ryan,
\emph{Introduction to Tensor Products of Banach Spaces},
Springer Monographs in Mathematics,
Springer, London, 2002.

\bibitem{Se}
Z. Semadeni,
\emph{Banach Spaces of Continuous Functions, Vol.~I},
Monografie Matematyczne 55,
PWN, Warsaw, 1971.

\bibitem{Se2}
Z. Semadeni,
\emph{Banach spaces non-isomorphic to their Cartesian squares II},
Bull. Polish Acad. Sci. Math. \textbf{8} (1960), 81--84.

\bibitem{Singer}
I. Singer,
\emph{Best Approximation in Normed Linear Spaces by Elements of Linear Subspaces},
Springer, Berlin--New York, 1970.

\bibitem{AvilesTodorcevicLexicographic}
A.~Avil\'es and S.~Todorcevic,
\emph{Lexicographic products as compact spaces of the first Baire class},
Topology Appl. \textbf{267} (2019), Art.~106871, 7~pp.

\bibitem{Tod}
S. Todor\v{c}evi\'c,
\emph{Trees and linearly ordered sets},
in \emph{Handbook of Set-Theoretic Topology},
K. Kunen and J. E. Vaughan (eds.),
North-Holland, Amsterdam, 1984, pp.~235--293.

\end{thebibliography}
\end{document}